\documentclass[final]{siamltex}

\usepackage{graphicx}
\usepackage{cite}
\usepackage{verbatim}
\usepackage{amssymb,amsfonts,amsmath,amscd}
\usepackage{mathrsfs}

\usepackage[dvipsnames,usenames]{color}
\usepackage[colorlinks,%
linkcolor=BrickRed,%
filecolor =BrickRed,%
citecolor=RoyalPurple,%
]{hyperref}

\def\spm#1{\typeout{!}}

\def\sU{{\mathsf U}}
\def\sX{{\mathsf X}}
\def\sY{{\mathsf Y}}
\def\sZ{{\mathsf Z}}

\def\cB{{\mathcal B}}

\def\cF{{\mathcal F}}
\def\cK{{\mathcal K}}
\def\cL{{\mathcal L}}
\def\cM{{\mathcal M}}
\def\cP{{\mathcal P}}

\def\law{\operatorname{Law}}

\def\d{\operatorname{d}\!}
\def\E{{\mathbb E}}
\def\Pr{{\mathbb P}}
\def\1{{\mathsf 1}}
\def\deq{\triangleq}
\def\bd#1{\boldsymbol{#1}}

\def\R{{\mathbb R}}

\def\argmin{\operatornamewithlimits{arg\,min}}
\def\essinf{\operatornamewithlimits{ess\,inf}}

\def\cI{\mathcal{I}}

\def\tPhi{{\Phi}}
\def\eps{\varepsilon}

\newtheorem{remark}{Remark}
\newtheorem{example}{Example}

\title{ Rationally Inattentive Control\\
of Markov Processes\thanks{This 
        work was supported in part by the NSF under award nos.\ CCF-1254041, CCF-1302438, ECCS-1135598, by the Center for Science of Information (CSoI), an NSF Science and Technology Center, under grant agreement CCF-0939370, and in part by the UIUC
		College of Engineering under Strategic Research Initiative on ``Cognitive
		and Algorithmic Decision Making." The material in this paper was presented in part at the 2013 American Control Conference and at the 2013 IEEE Conference on Decision and Control.}}

		\newcommand*\samethanks[1][\value{footnote}]{\footnotemark[#1]}
		\author{Ehsan Shafieepoorfard\thanks{Coordinated Science Laboratory and Department of Electrical and Computer Engineering, University of Illinois at Urbana-Champaign (shafiee1@illinois.edu, maxim@illinois.edu).}
		\and Maxim Raginsky\samethanks
		\and Sean P.~Meyn\thanks{Laboratory for Cognition \& Control in Complex Systems, Department of Electrical and Computer Engineering, University of Florida (meyn@ece.ufl.edu).}}

\begin{document}

\maketitle

\begin{abstract}
The article poses a general model for optimal control subject to information constraints, motivated in part by recent work of Sims and others on information-constrained decision-making by economic agents. In the average-cost optimal control framework, the general model introduced in this paper reduces to a variant of the linear-programming representation of the average-cost optimal control problem, subject to an additional mutual information constraint on the randomized stationary policy. The resulting optimization problem is convex and admits a decomposition based on the Bellman error, which is the object of study in approximate dynamic programming. The theory is illustrated through the example of information-constrained linear-quadratic-Gaussian (LQG) control problem. Some results on the infinite-horizon discounted-cost criterion are also presented.
\end{abstract}

\begin{keywords} 
Stochastic control, information theory, observation channels, optimization, Markov decision processes
\end{keywords}

\begin{AMS}
94A34, 90C40, 90C47
\end{AMS}

\pagestyle{myheadings}
\thispagestyle{plain}
\markboth{SHAFIEEPOORFARD, RAGINSKY, MEYN}{RATIONALLY INATTENTIVE CONTROL OF MARKOV PROCESSES}


\section{Introduction}
\label{intro}

The problem of {\itshape optimization with imperfect information} \cite{Bertsekas} deals with situations where a decision maker (DM) does not have direct access to the exact value of a payoff-relevant variable. Instead, the DM receives a noisy signal pertaining to this variable and makes decisions conditionally on that signal. 

It is usually assumed that the observation channel that delivers the signal is fixed {\itshape a priori}. In this paper, we do away with this assumption and investigate a class of dynamic optimization problems, in which the DM is free to choose the observation channel from a certain convex set. This formulation is inspired by the framework of {\itshape Rational Inattention}, proposed by the well-known economist Christopher Sims\footnote{Christopher Sims has shared the 2011 Nobel Memorial Prize in Economics with Thomas Sargent.} to model decision-making by agents who minimize expected cost given available information (hence ``rational''), but are capable of handling only a limited amount of information (hence ``inattention'') \cite{sims2003implications, sims2006rational}.  Quantitatively, this limitation is stated as an upper bound on the \textit{mutual information} in the sense of Shannon \cite{PinskerBook} between the state of the system and the signal available to the DM. 

Our goal in this paper is to initiate the development of a general theory of optimal control subject to mutual information constraints.  We focus on the  average-cost optimal control problem for Markov processes and show that the construction of an optimal information-constrained control law reduces to a variant of the linear-programming representation of the average-cost optimal control problem, subject to an additional mutual information constraint on the randomized stationary policy. The resulting optimization problem is convex and admits a decomposition in terms of the Bellman error, which is the object of study in approximate dynamic programming \cite{meyn2008control,Bertsekas}. This decomposition reveals a fundamental connection between information-constrained controller design and \textit{rate-distortion theory} \cite{Ratedistortion}, a branch of information theory that deals with optimal compression of data subject to information constraints. 

Let us give a brief informal sketch of the problem formulation; precise definitions and regularity/measurability assumptions are spelled out in the sequel. Let $\sX$, $\sU$, and $\sZ$ denote the state, the control (or action), and the observation spaces. The objective of the DM is to control a discrete-time state process $\{X_t\}^\infty_{t=1}$ with values in $\sX$ by means of a randomized control law (or policy) $\Phi(\d u_t|z_t)$, $t \ge 1$, which generates a random action $U_t \in \sU$ conditionally on the observation $Z_t \in \sZ$. The observation $Z_t$, in turn, depends stochastically on the current state $X_t$ according to an observation model (or information structure) $W(\d z_t|x_t)$. Given the current action $U_t=u_t$ and the current state $X_t = x_t$, the next state $X_{t+1}$ is determined by the state transition law $Q(\d x_{t+1}|x_t,u_t)$. Given a one-step state-action cost function $c : \sX \times \sU \to \R^+$ and the initial state distribution $\mu = \law(X_1)$, the \textit{pathwise long-term average cost} of any pair $(\Phi,W)$ consisting of a policy and an observation model is given by
\begin{align*}
J_\mu(\Phi,W) \deq 
	\limsup_{T \to \infty} \frac{1}{T}
	 \sum^{T}_{t=1} c(X_t,U_t),
\end{align*}
where the law of the process $\{(X_t,Z_t,U_t)\}$ is induced by the pair $(\Phi,W)$ and by the law $\mu$ of $X_1$; for notational convenience, we will suppress the dependence on the fixed state transition dynamics $Q$.

 If the information structure $W$ is fixed, then we have a Partially Observable Markov Decision Process, where the objective of the DM is to pick a policy $\Phi^*$ to minimize $J_\mu(\Phi,W)$. In the framework of rational inattention, however, the DM is also allowed to optimize the choice of the information structure $W$ subject to a mutual information constraint. Thus, the DM faces the following optimization problem:\footnote{Since $J_\mu(\Phi,W)$ is a random variable that depends on the entire path $\{(X_t,U_t)\}$, the definition of a minimizing pair $(\Phi,W)$ requires some care. The details are spelled out in Section~\ref{sec:formulation}.}
\begin{subequations}\label{eq:baby_RI2}
\begin{align}
\text{minimize\,\,} & J_\mu(\Phi,W) \label{eq:baby_RI_AC2}\\
\text{subject to\,\,} &\limsup_{t \to \infty} I(X_t;Z_t) \le R \label{eq:baby_RI_info2}
\end{align}
\end{subequations}
where $I(X_t; Z_t)$ denotes the Shannon mutual information between the state and the observation at time $t$, and $R \ge 0$ is a given constraint value. The mutual information quantifies the amount of statistical dependence between $X_t$ and $Z_t$; in particular, it is equal to zero if and only if $X_t$ and $Z_t$ are independent, so the limit $R \to 0$ corresponds to open-loop policies. If $R < \infty$, then the act of generating the observation $Z_t$ will in general involve loss of information about the state $X_t$ (the case of perfect information corresponds to taking $R \to \infty$). However, for a given value of $R$, the DM is allowed to optimize the observation model $W$ and the control law $\Phi$ jointly to make the best use of all available information. In light of this, it is also reasonable to grant the DM the freedom to optimize the choice of the \textit{observation space} $\sZ$, i.e., to choose the optimal {\em representation} for the data supplied to the controller. In fact, it is precisely this additional freedom that enables the reduction of the rationally inattentive optimal control problem to an infinite-dimensional convex program.

This paper addresses the following problems: (a) give existence results for optimal information-constrained control policies; (b) describe the structure of such policies; and (c) derive an information-constrained analogue of the Average-Cost Optimality Equation (ACOE). Items (a) and (b) are covered by Theorem~\ref{thm:RI_DRF}, whereas Item (c) is covered by Theorem~\ref{thm:converse} and subsequent discussion in Section~\ref{ssec:implication}. We will illustrate the general theory through the specific example of an information-constrained Linear Quadratic Gaussian (LQG) control problem. Finally, we will outline an extension of our approach to the more difficult infinite-horizon discounted-cost case.

\subsection{Relevant literature}

In the economics literature, the rational inattention model has been used to explain certain memory effects in different economic equilibria \cite{stickiness}, to model various situations such as portfolio selection \cite{portfolio} or Bayesian learning \cite{Peng2005}, and to address some puzzles in macroeconomics and finance \cite{homebias, underdiver , stickyprices}. However, most of these results rely on heuristic considerations or on simplifying assumptions pertaining to the structure of observation channels.

 On the other hand, dynamic optimization problems where the DM observes the system state through an information-limited channel have been long studied by control theorists (a very partial list of references is \cite{Varaiya_Walrand_causal_control,bansal1989simultaneous,bacsar1994optimum,Tatikonda_Mitter_noisy_channels,Tatikonda_Sahai_Mitter_SLC,borkar2001markov,yuksel2012optimization}). Most of this literature focuses on the case when the channel is fixed, and the controller must be supplemented by a suitable encoder/decoder pair respecting the information constraint and any considerations of causality and delay. Notable exceptions include classic results of Bansal and Ba\c{s}ar \cite{bansal1989simultaneous,bacsar1994optimum} and recent work of Y\"uksel and Linder \cite{yuksel2012optimization}. The former is concerned with a linear-quadratic-Gaussian (LQG) control problem, where the DM must jointly optimize a linear observation channel and a control law to minimize expected state-action cost, while satisfying an average power constraint; information-theoretic ideas are used to simplify the problem by introducing a certain sufficient statistic. The latter considers a general problem of selecting optimal observation channels in static and dynamic stochastic control problems, but focuses mainly on abstract structural results pertaining to existence of optimal channels and to continuity of the optimal cost in various topologies on the space of observation channels.

The paper is organized as follows: The next section introduces the notation and the necessary information-theoretic preliminaries. Problem formulation is given in Section~\ref{sec:formulation}, followed by a brief exposition of rate-distortion theory in Section~\ref{sec:one_stage}. In Section~\ref{sec:convex_analytic}, we present our analysis of the problem via a synthesis of rate-distortion theory and the convex-analytic approach to Markov decision processes (see, e.g., \cite{Borkar_convex_analytic}). We apply the theory to an information-constrained variant of the LQG control problem in Section~\ref{sec:IC_LQG}. All of these results pertain to the average-cost criterion; the more difficult infinite-horizon discounted-cost criterion is considered in Section~\ref{sec:discounted}. Certain technical and auxiliary results are relegated to Appendices.

Preliminary versions of some of the results were reported in \cite{EhsanACC} and \cite{EhsanCDC}.


\section{Preliminaries and notation}
\label{sec:preliminar}

All spaces are assumed to be standard Borel (i.e., isomorphic to a Borel subset of a complete separable metric space); any such space will be equipped with its Borel $\sigma$-field $\cB(\cdot)$. We will repeatedly use standard notions results from probability theory, as briefly listed below; we refer the reader to the text by Kallenberg \cite{Kallenberg} for details. The space of all probability measures on $(\sX,\cB(\sX))$ will be denoted by $\cP(\sX)$; the sets of all measurable functions and all bounded continuous functions $\sX \to \R$ will be denoted by $M(\sX)$ and by $C_b(\sX)$, respectively. We use the standard linear-functional notation for expectations: given an $\sX$-valued random object $X$ with $\law(X) = \mu \in \cP(\sX)$ and $f \in L^1(\mu) \subset M(\sX)$,
\begin{align*}
	\langle \mu,f \rangle \deq \int_\sX f(x) \mu(\d x) = \E[f(X)].
\end{align*}
A \textit{Markov} (or \textit{stochastic}) \textit{kernel} with input space $\sX$ and output space $\sY$ is a mapping $K(\cdot|\cdot) : \cB(\sY) \times \sX \to [0,1]$, such that $K(\cdot|x) \in \cP(\sY)$ for all $x \in \sX$ and $x \mapsto K(B|x) \in M(\sX)$ for every $B \in \cB(\sY)$. We denote the space of all such kernels by $\cM(\sY|\sX)$. Any $K \in \cM(\sY|\sX)$ acts on $f \in \cM(\sY)$ from the left and on $\mu \in \cP(\sX)$ from the right:
\begin{align*}
Kf(\cdot) \deq \int_\sY f(y) K(\d y|\cdot), \qquad \mu K(\cdot) \deq \int_{\sX} K(\cdot|x)\mu(\d x).
\end{align*}
Note that $Kf \in \cM(\sX)$ for any $f \in \cM(\sY)$, and $\mu K \in \cP(\sY)$ for any $\mu \in \cP(\sX)$.  Given a probability measure $\mu \in \cP(\sX)$, and a Markov kernel $K \in \cM(\sY|\sX)$, we denote by $\mu \otimes K$ a probability measure defined on the product space $(\sX \times \sY, \cB(\sX) \otimes \cB(\sY))$ via its action on the rectangles $A \times B$, $A \in \cB(\sX), B \in \cB(\sY)$:
\begin{align*}
	(\mu \otimes K)(A \times B) \deq \int_A K(B|x)\mu(\d x).
\end{align*}
If we let $A = \sX$ in the above definition, then we end up with with $\mu K(B)$. Note that product measures $\mu \otimes \nu$, where $\nu \in \cP(\sY)$, arise as a special case of this construction, since any $\nu \in \cP(\sY)$ can be realized as a Markov kernel $(B,x) \mapsto \nu(B)$.

We also need some notions from information theory. The {\em relative entropy} (or {\em information divergence}) \cite{PinskerBook} between any two probability measures $\mu,\nu \in \cP(\sX)$ is
\begin{align*}
	D(\mu \| \nu) &\deq \begin{cases}
	\left\langle \mu, \log \displaystyle\frac{\d\mu}{\d\nu} \right\rangle, & \text{ if $\mu \prec \nu$} \\
	+\infty, & \text{otherwise}
\end{cases}
\end{align*}
where $\prec$ denotes absolute continuity of measures, and $\d\mu/\d\nu$ is the Radon--Nikodym derivative. It is always nonnegative, and is equal to zero if and only if $\mu \equiv \nu$. The {\em Shannon mutual information} \cite{PinskerBook} in $(\mu,K) \in \cP(\sX) \times \cM(\sY|\sX)$ is
\begin{equation}
	I(\mu,K) \deq D(\mu \otimes K \| \mu \otimes \mu K), 
\label{e:ShannonI}
\end{equation}
The functional $I(\mu,K)$ is concave in $\mu$, convex in $K$, and weakly lower semicontinuous in the joint law $\mu \otimes K$: for any two sequences $\{\mu_n\}^\infty_{n=1} \subset \cP(\sX)$ and $\{K_n\}^\infty_{n=1} \subset \cM(\sY|\sX)$ such that $\mu_n \otimes K_n \xrightarrow{n \to \infty} \mu \otimes K$ weakly, we have
\begin{align}\label{eq:MI_LSC}
\liminf_{n \to \infty} I(\mu_n,K_n) \ge I(\mu,K)
\end{align}
(indeed, if $\mu_n \otimes K_n$ converges to $\mu \otimes K$ weakly, then, by considering test functions in $C_b(\sX)$ and $C_b(\sY)$, we see that $\mu_n \to \mu$ and $\mu_n K_n \to \mu K$ weakly as well; Eq.~\eqref{eq:MI_LSC} then follows from the fact that the relative entropy is weakly lower-semicontinuous in both of its arguments \cite{PinskerBook}). If $(X,Y)$ is a pair of random objects with $\law(X,Y) = \Gamma = \mu \otimes K$, then we will also write $I(X; Y)$ or $I(\Gamma)$ for $I(\mu, K)$. In this paper, we use natural logarithms, so mutual information is measured in \textit{nats}. The mutual information admits the following variational representation \cite{infotheory}:
 \begin{align} \label{variatee}
	I(\mu,K) = \inf_{\nu \in \cP(\sY)} D(\mu \otimes K \| \mu \otimes \nu),
\end{align}
where the infimum is achieved by $\nu = \mu K$. It also satisfies an important relation known as the \textit{data processing inequality}: Let $(X,Y,Z)$ be a triple of jointly distributed random objects, such that $X$ and $Z$ are conditionally independent given $Y$. Then
\begin{align}\label{eq:DPI}
	I(X;Z) \le I(X;Y).
\end{align}
In words, no additional processing can increase information.


\section{Problem formulation and simplification}
\label{sec:formulation}

\begin{figure}
	\begin{center}
		\includegraphics[width=0.4\columnwidth]{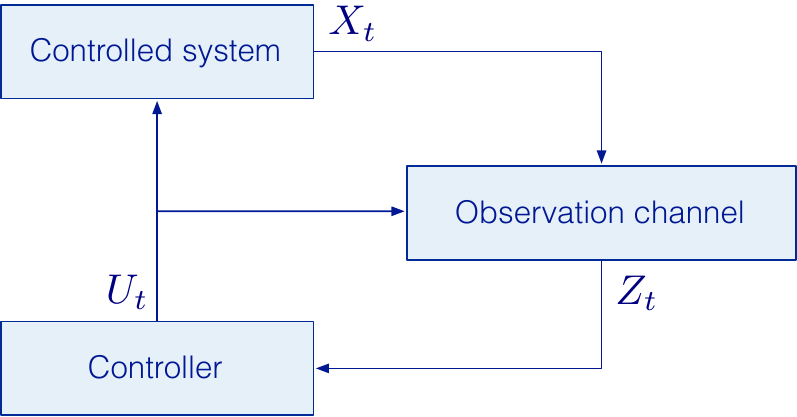}
	\end{center}
	\caption{\label{fig:system}System model.}
\end{figure}

We now give a more precise formulation for the problem \eqref{eq:baby_RI2} and take several simplifying steps towards its solution. We consider a model with a block diagram shown in Figure~\ref{fig:system},  where the DM is constrained to observe the state of the controlled system through an information-limited channel. The model is fully specified by the following ingredients:
\begin{itemize}
	\item[({\bf M.1})] the state, observation and control spaces denoted by $\sX$, $\sZ$ and $\sU$ respectively;
	\item[({\bf M.2})] the (time-invariant) controlled system, specified by a stochastic kernel $Q \in \cM(\sX | \sX \times \sU)$ that describes the dynamics of the system state, initially distributed according to $\mu \in \cP(\sX)$;
	\item[({\bf M.3})] the observation channel, specified by a	stochastic kernel $W \in \cM(\sZ | \sX)$;
	\item[({\bf M.4})] the feedback controller, specified by a stochastic kernel $\Phi \in \cM(\sU | \sZ)$.
\end{itemize}
The $\sX$-valued state process $\{X_t\}$, the $\sZ$-valued observation process $\{Z_t\}$, and the $\sU$-valued control $\{U_t\}$ process are realized on the canonical path space $(\Omega,\cF,\Pr^{W,\Phi}_\mu)$, where $\Omega = \sX^{{\mathbb N}} \times \sZ^{{\mathbb N}} \times \sU^{{\mathbb N}}$, $\cF$ is the Borel $\sigma$-field of $\Omega$, and for every $t \ge 1$
\begin{align*}
	X_t(\omega) = x(t), \quad Z_t(\omega) = z(t), \quad U_t(\omega) = u(t)
\end{align*}
with $\omega = (x,z,u) = \big((x(1),x(2),\ldots,),(z(1),z(2),\ldots),(u(1),u(2),\ldots)\big)$. The process distribution satisfies
 $\Pr^{W,\Phi}_\mu(X_1 \in \cdot) = \mu$, and
\begin{align*}
&\Pr^{W,\Phi}_\mu(Z_t \in \cdot | X^t,Z^{t-1},U^{t-1}) = W(\cdot | X_t) \\
&\Pr^{W,\Phi}_\mu(U_t \in \cdot | X^t,Z^t,U^{t-1}) = \Phi(\cdot|Z_t) \\
&\Pr^{W,\Phi}_\mu(X_{t+1} \in \cdot | X^t,Z^t,U^t) = Q(\cdot|X_t,U_t).
\end{align*}
Here and elsewhere, $X^t$ denotes the tuple $(X_1,\ldots,X_t)$; the same applies to $Z^t$, $U^t$, etc. This specification ensures that, for each $t$, the next state $X_{t+1}$ is conditionally independent of $X^{t-1},Z^t,U^{t-1}$ given $X_t,U_t$ (which is the usual case of a controlled Markov process), that the control $U_t$ is conditionally independent of $X^t,Z^{t-1},U^{t-1}$ given $Z_t$, and that the observation $Z_t$ is conditionally independent of $X^{t-1},Z^{t-1},U^{t-1}$ given the most recent state $X_t$. In other words, at each time $t$ the controller takes as input only the most recent observation $Z_t$, which amounts to the assumption that there is a {\em separation structure} between the observation channel and the controller. This assumption is common in the literature\cite{Varaiya_Walrand_causal_control,Tatikonda_Mitter_noisy_channels,Tatikonda_Sahai_Mitter_SLC}. We also assume that the observation $Z_t$ depends only on the current state $X_t$; this assumption appears to be rather restrictive, but, as we show in Appendix~\ref{ssec:standardproblem}, it entails no loss of generality under the above separation structure assumption.

We now return to the information-constrained control problem stated in Eq.~\eqref{eq:baby_RI2}. If we fix the observation space $\sZ$, then the problem of finding an optimal pair $(W,\Phi)$ is difficult even in the single-stage $(T=1)$ case. Indeed, if we fix $W$, then the Bayes-optimal choice of the control law $\Phi$ is to minimize the expected posterior cost:
\begin{align*}
	\Phi^*_W(\d u |z) &= \delta_{u^*(z)}(\d u), \quad \text{where } u^*(z) = \argmin_{u \in \sU} \E[c(X,u)|Z=z].
\end{align*}
Thus, the problem of finding the optimal $W^*$ reduces to minimizing the functional
\begin{align*}
	W \mapsto \inf_{\Phi \in \cM(\sU|\sZ)} \int_{\sX \times \sU \times \sZ} \mu(\d x) W(\d z|x) \Phi(\d u |z) c(x,u)
\end{align*}
over the convex set $\left\{ W \in \cM(\sZ|\sX) : I(\mu, W) \le R \right\}$. However, this functional is concave, since it is given by a pointwise infimum of affine functionals. Hence, the problem of jointly optimizing $(W,\Phi)$ for a fixed observation space $\sZ$ is nonconvex even in the simplest single-stage setting. This lack of convexity is common in control problems with ``nonclassical'' information structures \cite{kulkarni2012optimizer}.

Now, from the viewpoint of rational inattention, the objective of the DM is to make the \textit{best} possible use of all available information subject only to the mutual information constraint. From this perspective, fixing the observation space $\sZ$ could be interpreted as suboptimal. Indeed, we now show that if we allow the DM an additional freedom to choose $\sZ$, and not just the information structure $W$, then we may simplify the problem by collapsing the three decisions of choosing $\sZ,W,\Phi$ into one of choosing a Markov randomized stationary (MRS) control law $\Phi \in \cM(\sU|\sX)$ satisfying the information constraint $\limsup_{t \to \infty} I(\mu_t,\Phi) \le R$, where $\mu_t = \Pr^{\Phi}_\mu(X_t \in \cdot)$ is the distribution of the state at time $t$, and $\Pr^\Phi_\mu$ denotes the process distribution of $\{(X_t,U_t)\}^\infty_{t=1}$, under which $\Pr^\Phi_\mu(X_1 \in \cdot) = \mu$, $\Pr^\Phi_\mu(U_t \in \cdot|X^t,U^{t-1}) = \Phi(\cdot|X_t)$, and $\Pr^\Phi_\mu(X_{t+1} \in \cdot|X^t,U^t) = Q(\cdot|X_t,U_t)$.  Indeed, fix an arbitrary triple $(\sZ,W,\Phi)$, such that the information constraint \eqref{eq:baby_RI_info2} is satisfied w.r.t. $\Pr^{W,\Phi}_\mu$:
\begin{align}\label{eq:Z_fixed_info}
\limsup_{t \to \infty} I(X_t; Z_t) \le R.
\end{align}
Now consider a new triple $(\sZ', W', \Phi')$ with $\sZ' = \sU$, $W' = \Phi \circ W$, and $\Phi'(\d u|z) = \delta_z(\d u)$, where $\delta_z$ is the Dirac measure centered at $z$. Then obviously $\Pr((X_t,U_t) \in \cdot)$ is the same in both cases, so that $J_\mu(\Phi',W') = J_\mu(\Phi,W)$. On the other hand, from \eqref{eq:Z_fixed_info} and from the data processing inequality \eqref{eq:DPI} we get
\begin{align*}
	\limsup_{t \to \infty} I(\mu_t, W') = \limsup_{t \to \infty} I(\mu_t, \Phi \circ W) \le \limsup_{t \to \infty} I(\mu_t, W) \le R,
\end{align*}
so the information constraint is still satisfied. Conceptually, this reduction describes a DM who receives perfect information about the state $X_t$, but must discard some of this information ``along the way'' to satisfy the information constraint.

In light of the foregoing observations, from now on we let $Z_t = X_t$ and focus on the following information-constrained optimal control problem:
\begin{subequations}\label{eq:RI}
\begin{align}
	\text{minimize } & J_\mu(\Phi) \deq \limsup_{T \to \infty} \frac{1}{T} \sum^{T}_{t=1} c(X_t,U_t)
	\label{eq:objective}
	\\
	\text{subject to } & 
	\limsup_{t \to \infty} I(\mu_t,\Phi) \le R. \label{eq:info_constraint}
\end{align}
\end{subequations}
Here, the limit supremum in \eqref{eq:objective} is a random variable that depends on the entire path $\{(X_t,U_t)\}^\infty_{t=1}$, and the precise meaning of the minimization problem in \eqref{eq:objective} is as follows: We say that an MRS control law $\Phi^*$ satisfying the information constraint \eqref{eq:info_constraint} is optimal for \eqref{eq:objective} if 
\begin{align}\label{eq:RI_optimality}
J_\mu(\Phi^*) = \inf_{\Phi} \bar{J}_\mu(\Phi), \qquad \Pr^{\Phi^*}_\mu\text{-a.s.}
\end{align}
where
\begin{align}\label{eq:EAC}
	\bar{J}_\mu(\Phi) \deq \limsup_{T \to \infty}\frac{1}{T}\E^\Phi_\mu\left[\sum^T_{t=1}c(X_t,U_t)\right]
\end{align}
is the long-term expected average cost of MRS $\Phi$ with initial state distribution $\mu$, and
where the infimum on the right-hand side of Eq.~\eqref{eq:RI_optimality} is over all MRS control laws $\Phi$ satisfying the information constraint \eqref{eq:info_constraint} (see, e.g., \cite[p.~116]{HernandezLasserreMCPbook} for the definition of pathwise average-cost optimality in the information-unconstrained setting). However, 
we will see that, under general conditions, $J_\mu(\Phi^*)$ is deterministic  and independent of the initial condition.  


\section{One-stage problem: solution via rate-distortion theory}
\label{sec:one_stage}

Before we analyze the average-cost problem \eqref{eq:RI}, we show that the one-stage case can be solved completely using {\em rate-distortion theory} \cite{Ratedistortion} (a branch of information theory that deals with optimal compression of data subject to information constraints). Then, in the following section, we will tackle \eqref{eq:RI} by reducing it to a suitable one-stage problem.

With this in mind, we consider the following problem: 
\begin{subequations}\label{aval}
	\begin{align}
		\text{minimize }	& \langle \mu \otimes \Phi, c \rangle \\
		\text{subject to } & \Phi \in \cI_\mu(R)
	\end{align}
\end{subequations}
for a given probability measure $\mu \in \cP(\sX)$ and a given $R \ge 0$, where
\begin{align}\label{eq:DRF_feasible_set}
	\cI_\mu(R) \deq \Big\{ \Phi \in \cM(\sU|\sX) : I(\mu,\Phi) \le R \Big\}.
\end{align}
The set $\cI_\mu(R)$ is nonempty for every $R \ge 0$.  To see this, note that any kernel $\Phi_\diamond \in \cM(\sU|\sX)$ for which the function $x \mapsto \Phi_\diamond(B|x)$ is constant ($\mu$-a.e.\ for any $B \in \cB(\sU)$) satisfies $I(\mu,\Phi_\diamond) = 0$. Moreover, this set is convex since the functional $\Phi \mapsto I(\mu,\Phi)$ is convex for any fixed $\mu$. Thus, the optimization problem \eqref{aval} is convex, and its value is called the {\em Shannon distortion-rate function} (DRF) of $\mu$:
\begin{align}\label{eq:DRF}
	D_\mu(R;c) \deq \inf_{\Phi \in \cI_\mu(R)} \langle \mu \otimes \Phi, c \rangle.
\end{align}

In order to study the existence and the structure of a control law that achieves the infimum in \eqref{eq:DRF}, it is convenient to introduce the Lagrangian relaxation
\begin{align*}
	\cL_\mu(\Phi,\nu,s) \deq s D(\mu \otimes \Phi \| \mu \otimes \nu) + \langle \mu \otimes \Phi, c \rangle, \qquad s \ge 0, \nu \in \cP(\sU).
\end{align*}
From the variational formula \eqref{variatee} and the definition \eqref{eq:DRF} of the DRF it follows that
\begin{align*}
	\inf_{\Phi \in \cI_\mu(R)} 	\inf_{\nu \in \cP(\sU)} \cL_\mu(\Phi,\nu,s) \le s R + D_\mu(R; c).
\end{align*}
Then we have the following key result \cite{Csiszar_rate_distortion}:

\begin{proposition}\label{prop:DRF} The DRF $D_\mu(R;c)$ is convex and nonincreasing in $R$. Moreover, assume the following:
	\begin{itemize}
		\item[{\em({\bf D.1})}] The cost function $c$ is lower semicontinuous, satisfies
		$$
		\inf_{u \in \sU} c(x,u) > -\infty, \qquad \forall x \in \sX
		$$
		and is also {\em coercive}: there exist two sequences of compact sets $\sX_n \uparrow \sX$ and $\sU_n \uparrow \sU$ such that
		\begin{align*}
			\lim_{n \to \infty} \inf_{x \in \sX^c_n, u \in \sU^c_n} c(x,u) = + \infty.
		\end{align*}
		\item[{\em({\bf D.2})}] There exists some $u_0 \in \sU$ such that $\langle \mu, c(\cdot, u_0) \rangle < \infty$.
	\end{itemize}
Define the \textit{critical rate}
$$
R_0 \deq \inf \left\{ R \ge 0: D_\mu(R;c) = \Big\langle\mu, \inf_{u \in \sU} c(\cdot,u)\Big\rangle \right\}
$$
(it may take the value $+\infty$). Then, for any $R < R_0$ there exists a Markov kernel $\Phi^* \in \cM(\sU|\sX)$ satisfying $I(\mu,\Phi^*) = R$ and $\langle \mu \otimes \Phi^*, c \rangle = D_\mu(R;c)$. Moreover, the Radon--Nikodym derivative of the joint law $\mu \otimes \Phi^*$ w.r.t.\ the product of its marginals satisfies
	\begin{align}\label{eq:DRF_optimal_kernel}
		\frac{\d\,(\mu \otimes \Phi^*)}{\d\,(\mu \otimes \mu \Phi^*)}(x,u) = \alpha(x) e^{-\frac{1}{s}c(x,u)}
	\end{align}
	where $\alpha : \sX \to \R^+$ and $s \ge 0$ are such that
	\begin{align}\label{eq:DRF_optimality}
		\int_\sX \alpha(x)e^{-\frac{1}{s}c(x,u)}\mu(\d x) \le 1, \qquad \forall u \in \sU
	\end{align}
	and $-s$ is the slope of a line tangent to the graph of $D_\mu(R;c)$ at $R$:
	\begin{align}\label{eq:tangent_DRF}
		D_\mu(R';c) + s R' \ge D_\mu(R;c) + s R, \qquad \forall R' \ge 0.
	\end{align}
For any $R \ge R_0$, there exists a Markov kernel $\Phi^* \in \cM(\sU|\sX)$ satisfying
$$
\langle \mu \otimes \Phi^*, c \rangle = \Big\langle \mu, \inf_{u \in \sU} c(\cdot,u) \Big\rangle
$$
and $I(\mu, \Phi^*) = R_0$. This Markov kernel is deterministic, and is implemented by $\Phi^*(\d u | x) = \delta_{u^*(x)}(\d u)$, where $u^*(x)$ is any minimizer of $c(x,u)$ over $u$. 
\end{proposition}

Upon substituting \eqref{eq:DRF_optimal_kernel} back into \eqref{eq:DRF} and using \eqref{eq:DRF_optimality} and \eqref{eq:tangent_DRF}, we get the following variational representation of the DRF:

\begin{proposition}\label{prop:DRF_dual} Under the conditions of Prop.~\ref{prop:DRF}, the DRF $D_\mu(R;c)$ can be expressed as
	\begin{align*}
		D_\mu(R;c) = \sup_{s \ge 0}\inf_{\nu \in \cP(\sU)} s\left[\left\langle \mu,\log\frac{1} {\int_\sU e^{-\frac{1}{s}c(\cdot,u)} \nu(\d u)}\right\rangle - R \right].
	\end{align*}
\end{proposition}


\section{Convex analytic approach for average-cost optimal control with rational inattention}
\label{sec:convex_analytic}

We now turn to the analysis of the average-cost control problem \eqref{eq:objective} with the information constraint \eqref{eq:info_constraint}. In multi-stage control problems, such as this one, the control law has a {\em dual effect} \cite{BarShalomTse_dual_effect}: it affects both the cost at the current stage and the uncertainty about the state at future stages. The presence of the mutual information constraint \eqref{eq:info_constraint} enhances this dual effect, since it prevents the DM from ever learning ``too much" about the state. This, in turn, limits the DM's future ability to keep the average cost low. These considerations suggest that, in order to bring rate-distortion theory to bear on the problem \eqref{eq:objective}, we cannot use the one-stage cost $c$ as the distortion function. Instead, we must modify it to account for the effect of the control action on future costs. As we will see, this modification leads to a certain stochastic generalization of the Bellman Equation.

\subsection{Reduction to single-stage optimization}
\label{ssec:reduction1}

We begin by reducing the dynamic optimization problem \eqref{eq:RI} to a particular static (single-stage) problem. Once this has been carried out, we will be able to take advantage of the results of Section~\ref{sec:one_stage}. The reduction is based on the so-called {\em convex-analytic} approach to controlled Markov processes \cite{Borkar_convex_analytic} (see also \cite{Manne_LP,BorkarPTRF,HernandezLasserreLP,meyn2008control}), which we briefly summarize here.

 Suppose that we have a Markov control problem with initial state distribution $\mu \in \cP(\sX)$ and controlled transition kernel $Q \in \cM(\sX|\sX \times \sU)$. Any MRS control law $\Phi$ induces a transition kernel $Q_\tPhi \in \cM(\sX | \sX)$ on the state space $\sX$:
\begin{align*}
	Q_\tPhi(A|x) \deq \int_\sU Q(A|x,u)\Phi(\d u |x), \quad \forall A \in \cB(\sX).
\end{align*}
We wish to find an MRS control law $\Phi^* \in \cM(\sU|\sX)$ that would minimize the long-term average cost $J_\mu(\Phi)$ simultaneously for all $\mu$. With that in mind, let
$$
J^* \deq \inf_{\mu \in \cP(\sX)} \inf_{\Phi \in \cM(\sU|\sX)} \bar{J}_\mu(\Phi),
$$
where $\bar{J}_\mu(\Phi)$ is the long-term expected average cost defined in Eq.~\eqref{eq:EAC}. Under certain regularity conditions, we can guarantee the existence of an MRS control law $\Phi^*$, such that $J_\mu(\Phi^*) = J^*$ $\Pr^{\Phi^*}_\mu$-a.s.\ for all $\mu \in \cP(\sX)$. Moreover, this optimizing control law is \textit{stable} in the following sense:

\begin{definition} \label{stabb} An MRS control law $\Phi \in \cM(\sU|\sX)$ is called {\em stable} if:
	\begin{itemize}
	\item There exists at least one probability measure $\pi \in \cP(\sX)$, which is invariant w.r.t.\ $Q_\tPhi$: $\pi = \pi Q_\tPhi$.
	\item The average cost $\bar{J}_{\pi}(\Phi)$ is finite, and moreover
	 \begin{align*}
		\bar{J}_{\pi}(\Phi) = \langle \Gamma_\tPhi, c \rangle = \int_{\sX \times \sU} c(x,u) \Gamma_\tPhi(\d x, \d u), \qquad \text{where } \Gamma_\Phi \deq \pi \otimes \Phi.
	\end{align*}
\end{itemize}
The subset of $\cM(\sU|\sX)$ consisting of all such stable control laws will be denoted by $\cK$.
\end{definition}

Then we have the following \cite[Thm.~5.7.9]{HernandezLasserreMCPbook}:

\begin{theorem}\label{thm:convex_analytic}  Suppose that the following assumptions are satisfied:
\begin{itemize}
	\item[{\em({\bf A.1})}] The cost function $c$ is nonnegative, lower semicontinuous, and coercive.
	\item[{\em({\bf A.2})}] The cost function $c$ is inf-compact, i.e., for every $x \in \sX$ and every $r \in \R$, the set $\{u \in \sU: c(x,u) \le r \}$ is compact.	
	\item[{\em({\bf A.3})}] The kernel $Q$ is weakly continuous, i.e., $Qf \in C_b(\sX \times \sU)$ for any $f \in C_b(\sX)$.
	\item[{\em({\bf A.4})}] There exist an MRS control law $\Phi$ and an initial state $x \in \sX$, such that $J_{\delta_x}(\Phi) < \infty$.
\end{itemize}
Then there exists a control law $\Phi^* \in \cK$, such that
	\begin{align}\label{eq:AC_SRS_optimality}
		J^* = \bar{J}_{\pi^*}(\Phi^*)  = \inf_{\Phi \in \cK} \langle \Gamma_\tPhi, c \rangle,
	\end{align}
	where $\pi^* = \pi^* Q_{\Phi^*}$. Moreover, if $\Phi^*$ is such that the induced kernel $Q^* = Q_{\Phi^*}$ is Harris-recurrent, then $J_\mu(\Phi^*) = J^*$ $\Pr^{\Phi^*}_\mu$-a.s.\ for all $\mu \in \cP(\sX)$.
\end{theorem}

One important consequence of the above theorem is that, if $\Phi^* \in \cK$ achieves the infimum on the rightmost side of \eqref{eq:AC_SRS_optimality} and if $\pi^*$ is the unique invariant distribution of the Harris-recurrent Markov kernel $Q_{\Phi^*}$, then the state distributions $\mu_t$ induced by $\Phi^*$ converge weakly to $\pi^*$ regardless of the initial condition $\mu_1 = \mu$. Moreover, the theorem allows us to focus on the \textit{static} optimization problem given by the right-hand side of Eq.~\eqref{eq:AC_SRS_optimality}.

Our next step is to introduce a steady-state form of the information constraint \eqref{eq:info_constraint} and then to use ideas from rate-distortion theory to attack the resulting optimization problem. The main obstacle to direct application of the results from Section~\ref{sec:one_stage} is that the state distribution and the control policy in \eqref{eq:AC_SRS_optimality} are coupled through the invariance condition $\pi_\Phi = \pi_\Phi Q_\Phi$.  However, as we show next, it is possible to decouple the information and the invariance constraints by introducing a function-valued Lagrange multiplier to take care of the latter.

\subsection{Bellman error minimization via marginal decomposition}
\label{ssec:bellman}

We begin by decomposing the infimum over $\Phi$ in \eqref{eq:AC_SRS_optimality} by first fixing the marginal state distribution $\pi \in \cP(\sX)$. To that end, for a given $\pi \in \cP(\sX)$, we consider the set of all stable control laws that leave it invariant (this set might very well be empty): $\cK_\pi \deq  \left\{ \Phi \in \cK: \pi = \pi Q_\tPhi \right\}$. In addition, for a given value $R \ge 0$ of the information constraint, we consider the set $\cI_\pi(R) = \left\{ \Phi \in \cM(\sU|\sX) : I(\pi,\Phi) \le R \right\}$ (recall Eq.~\eqref{eq:DRF_feasible_set}).

Assuming that the conditions of Theorem~\ref{thm:convex_analytic} are satisfied, we can rewrite the expected ergodic cost \eqref{eq:AC_SRS_optimality} (in the absence of information constraints) as
\begin{align}\label{eq:steady_state_RI0}
	J^* = \inf_{\Phi \in \cK} \langle \Gamma_\tPhi, c \rangle = \inf_{\pi \in \cP(\sX)} \inf_{\Phi \in \cK_\pi} \langle \pi \otimes \Phi, c \rangle.
\end{align}
In the same spirit, we can now introduce the following \textit{steady-state} form of the information-constrained control problem \eqref{eq:RI}:
\begin{align}\label{eq:steady_state_RI}
J^*(R) \deq \inf_{\pi \in \cP(\sX)} \inf_{\Phi \in \cK_\pi(R)} \langle \pi \otimes \Phi, c \rangle,
\end{align}
where the feasible set $\cK_\pi(R) \deq \cK_\pi \cap \cI_\pi(R)$ accounts for both the invariance constraint and the information constraint.

As a first step to understanding solutions to \eqref{eq:steady_state_RI}, we consider each candidate invariant distribution $\pi \in \cP(\sX)$ separately and define
\begin{align}\label{eq:simplerr2}
J^*_\pi(R) \deq \inf_{\Phi \in \cK_\pi(R)} 
\langle\pi \otimes \Phi, c\rangle
\end{align}
(we set the infimum to $+\infty$ if $\cK_\pi = \varnothing$). Now we follow the usual route in the theory of average-cost optimal control \cite[Ch.~9]{meyn2008control} and eliminate the invariance condition $\Phi \in \cK_\pi$ by introducing a function-valued Lagrange multiplier:

\begin{proposition}\label{prop:propaval} For any $\pi \in \cP(\sX)$,
\begin{align}\label{eq:J_pi}
J^*_\pi(R) = \inf_{\Phi \in \cI_\pi(R)} \sup_{h \in C_b(\sX)} \langle \pi \otimes \Phi, c + Qh - h \otimes \1  \rangle.
\end{align}
\end{proposition}
\begin{remark}{\em 
Both in \eqref{eq:J_pi} and elsewhere, we can extend the supremum over $h \in C_b(\sX)$ to all $h \in L^1(\pi)$ without affecting the value of $J^*_\pi(R)$ (see, e.g., the discussion of abstract minimax duality in \cite[App.~1.3]{villani2003})}.
\end{remark}

\begin{remark} {\em
Upon setting  $\lambda_\pi = J^*_\pi(R)$, we can recognize  the function $c + Qh - h\otimes \1 -\lambda_\pi$ as the  \textit{Bellman error} associated with $h$; this object plays a central role in approximate dynamic programming.}
\end{remark}

\begin{proof} 
	Let $\iota_{\pi}(\Phi)$ take the value $0$ if $\Phi \in \cK_\pi$ and $+\infty$ otherwise.  Then
	\begin{align}\label{eq:J_pi_alt}
	J^*_\pi(R) = \inf_{\Phi \in \cI_\pi(R)} \left[ \big\langle \pi \otimes \Phi, c \big\rangle  + \iota_{\pi}(\Phi)\right].
	\end{align}
	Moreover,
	\begin{align}\label{eq:iota_as_sup}
	\iota_\pi(\Phi) &= \sup_{h \in C_b(\sX)} \left[ \big\langle \pi Q_\tPhi, h \big\rangle - \big\langle \pi, h \big\rangle \right]
	\end{align}
	Indeed, if $\Phi \in K_\pi$, then the right-hand side of \eqref{eq:iota_as_sup} is zero. On the other hand, suppose that $\Phi \not\in K_\pi$. Since $\sX$ is standard Borel, any two probability measures $\mu,\nu \in \cP(\sX)$ are equal if and only if $\langle \mu,h\rangle = \langle \nu,h\rangle$ for all $h \in C_b(\sX)$. Consequently, $\langle \pi,h_0 \rangle \neq \langle \pi Q_\tPhi, h_0 \rangle$ for some $h_0 \in C_b(\sX)$. There is no loss of generality if we assume that $\langle \pi Q_\tPhi, h_0 \rangle - \langle \pi,h_0 \rangle > 0$. Then by considering functions $h^n_0 = nh_0$ for all $n = 1,2,\ldots$ and taking the limit as $n \to \infty$, we can make the right-hand side of \eqref{eq:iota_as_sup} grow without bound. This proves \eqref{eq:iota_as_sup}. Substituting it into \eqref{eq:J_pi_alt}, we get \eqref{eq:J_pi}.
\end{proof} 

Armed with this proposition, we can express \eqref{eq:steady_state_RI} in the form of an appropriate rate-distortion problem by fixing $\pi$ and considering the {\em dual value} for \eqref{eq:J_pi}:
\begin{align}\label{eq:RI_dual}
J_{*,\pi}(R) \deq \sup_{h \in C_b(\sX)} \inf_{\Phi \in \cI_\pi(R)} \langle \pi \otimes \Phi, c + Qh - h \otimes \1 \rangle.
\end{align}
\begin{proposition}\label{prop:strong_duality} Suppose that assumption {\em ({\bf A.1})} above is satisfied, and that $J^*_\pi(R) < \infty$. Then the primal value $J^*_\pi(R)$ and the dual value $J_{*,\pi}(R)$ are equal.
\end{proposition}

\begin{proof} Let $\cP^0_{\pi,c}(R) \subset \cP(\sX \times \sU)$ be the closure, in the weak topology, of the set of all  $\Gamma \in \cP(\sX \times \sU)$, such that $\Gamma(\cdot \times \sU) = \pi(\cdot)$, $I(\Gamma) \le R$,  and $\langle \Gamma,c \rangle \le J^*_\pi(R)$. Since $J^*_\pi(R) < \infty$ by hypothesis, we can write
		\begin{align}\label{eq:J_primal_equiv}
			J^*_\pi(R) = \inf_{\Gamma \in \cP^0_{\pi,c}(R)} \sup_{h \in C_b(\sX)} \langle \Gamma, c + Q h - h \otimes \1  \rangle
		\end{align}
		and
		\begin{align}\label{eq:J_dual_equiv}
			J_{*,\pi}(R) = \sup_{h \in C_b(\sX)} \inf_{\Gamma \in \cP^0_{\pi,c}(R)} \langle \Gamma, c + Q h - h \otimes \1 \rangle.
		\end{align}
	Because $c$ is coercive and nonnegative, and $J^*_\pi(R) < \infty$, the set $\{ \Gamma \in \cP(\sX \times \sU) : \langle \Gamma,c \rangle \le J^*_\pi(R) \}$ is tight \cite[Proposition~1.4.15]{HernandezLasserreMCbook}, so its closure is weakly sequentially compact by Prohorov's theorem. Moreover, because the function $\Gamma \mapsto I(\Gamma)$ is weakly lower semicontinuous \cite{PinskerBook}, the set $\{ \Gamma : I(\Gamma) \le R\}$ is closed. Therefore, the set $\cP^0_{\pi,c}(R)$ is closed and tight, hence weakly sequentially compact. Moreover, the sets $\cP^0_{\pi,c}(R)$ and $C_b(\sX)$ are both convex, and the objective function on the right-hand side of \eqref{eq:J_primal_equiv} is affine in $\Gamma$ and linear in $h$. Therefore, by Sion's minimax theorem \cite{Sion_minimax} we may interchange the supremum and the infimum to conclude that $J^*_\pi(R) = J_{*,\pi}(R)$.
\end{proof} 

We are now in a position to relate the optimal value $J^*_{\pi}(R) = J_{*,\pi}(R)$ to a suitable rate-distortion problem. Recalling the definition in Eq.~\eqref{eq:DRF}, for any $h \in C_b(\sX)$ we consider the DRF of $\pi$ w.r.t.\ the distortion function $c + Qh$:
\begin{align}\label{eq:DRFF}
	D_\pi(R; c + Qh) \deq \inf_{\Phi \in \cI_\pi(R)} \langle \pi \otimes \Phi, c + Qh \rangle.
\end{align}
We can now give the following structural result:

\begin{theorem}\label{thm:RI_DRF} Suppose that Assumptions~{\em(\textbf{A.1})--(\textbf{A.3})} of Theorem~\ref{thm:convex_analytic} are in force. Consider a probability measure $\pi \in \cP(\sX)$ such that $J^*_\pi(R) < \infty$, and the supremum over $h \in C_b(\sX)$ in \eqref{eq:RI_dual} is attained by some $h_\pi$. Define the critical rate
	$$
	R_{0,\pi} \deq  \min \left\{ R \ge 0: D_\pi(R; c + Qh_\pi) = \Big\langle\pi, \inf_{u \in \sU} \left[c(\cdot,u) + Qh_\pi(\cdot,u)\right]\Big\rangle \right\}.
$$ 
If $R < R_{0,\pi}$, then there exists an MRS control law $\Phi^* \in \cM(\sU|\sX)$ such that $I(\pi,\Phi^*) = R$, 
and the Radon--Nikodym derivative of $\pi \otimes \Phi^*$ w.r.t.\ $\pi \otimes \pi \Phi^*$ takes the form
	\begin{align}\label{eq:RI_optimal_kernel}
	\frac{\d\,(\pi \otimes \Phi^*)}{\d\,(\pi \otimes \pi \Phi^*)}(x,u) = \frac{ e^{-\frac{1}{s} d(x,u)}}{\int_\sU e^{-\frac{1}{s}d(x,u)}\pi\Phi^*(\d u)},
	\end{align}
	where $d(x,u) \deq c(x,u) + Qh_\pi(x,u)$, and $s \ge 0$ satisfies
	\begin{align}
	&	D_\pi(R'; c + Qh_\pi) + s R' \ge D_\pi(R; c + Q h_\pi) + s R, \qquad \forall R' \ge 0. \label{eq:tangent}
	\end{align}
If $R \ge R_{0,\pi}$, then the deterministic Markov policy $\Phi^*(\d u|x) = \delta_{u^*_\pi(x)} (\d u)$,
where $u^*_\pi(x)$ is any minimizer of $c(x,u) + Qh_\pi(x,u)$ over $u$, satisfies $I(\pi, \Phi^*) = R_{0,\pi}$. In both cases, we have
	\begin{align} \label{avalacoe}
		J^*_\pi(R) + \langle \pi , h_\pi\rangle = \langle \pi \otimes \Phi^*, c + Q h_\pi\rangle = D_{\pi}(R;c + Qh_\pi)  .
	\end{align}
Moreover, the optimal value $J^*_\pi(R)$ admits the following variational representation:
		\begin{align}\label{eq:RI_DRF_dddual}
		&	J^*_\pi(R)  = \sup_{s \ge 0}\sup_{h \in C_b(\sX)} \inf_{\nu \in \cP(\sU)} \Bigg\{ - \langle \pi, h \rangle \nonumber\\
		& \qquad \qquad \qquad + s \left[ \left\langle \pi,  \log\frac{1}{\int_\sU e^{-\frac{1}{s}[c(\cdot,u)+Qh(\cdot,u)]}\nu(\d u)} \right\rangle - R\right] \Bigg\}
		\end{align}
\end{theorem}

\begin{proof} Using Proposition~\ref{prop:strong_duality} and the definition \eqref{eq:RI_dual} of the dual value $J_{*,\pi}(R)$, we can express $J^*_\pi(R)$ as a pointwise supremum of a family of DRF's:
	\begin{align}\label{eq:sup_of_DRFs}
		J^*_\pi(R) = \sup_{h \in C_b(\sX)} \left[D_{\pi}(R;c + Qh) - \langle \pi, h \rangle \right].
	\end{align}
Since $J^*_\pi(R) < \infty$, we can apply Proposition~\ref{prop:DRF} separately for each $h \in C_b(\sX)$. Since $Q$ is weakly continuous by hypothesis, $Qh \in C_b(\sX \times \sU)$ for any $h \in C_b(\sX)$. In light of these observations, and owing to our hypotheses, we can ensure that Assumptions~({\bf D.1}) and ({\bf D.2}) of Proposition~\ref{prop:DRF} are satisfied. In particular, we can take $h_\pi \in C_b(\sX)$ that achieves the supremum in \eqref{eq:sup_of_DRFs} (such an $h_\pi$ exists by hypothesis) to deduce the existence of an MRS control law $\Phi^*$ that satisfies the information constraint with equality and achieves \eqref{avalacoe}. Using \eqref{eq:DRF_optimal_kernel} with
		\begin{align*}
			\alpha(x) = \frac{1}{\int_\sU e^{-\frac{1}{s}d(x,u)}\pi \Phi^*(\d u)},
		\end{align*}	
we obtain \eqref{eq:RI_optimal_kernel}. In the same way, \eqref{eq:tangent} follows from \eqref{eq:tangent_DRF} in Proposition~\ref{prop:DRF}. Finally, the variational formula \eqref{eq:RI_DRF_dddual} for the optimal value can be obtained immediately from \eqref{eq:sup_of_DRFs} and Proposition~\ref{prop:DRF_dual}.
\end{proof}

Note that the control law $\Phi^* \in \cM(\sU|\sX)$ characterized by Theorem~\ref{thm:RI_DRF} is not guaranteed to be feasible (let alone optimal) for the optimization problem in Eq.~\eqref{eq:simplerr2}. However, if we add the invariance condition $\Phi^* \in \cK_\pi$, then \eqref{avalacoe} provides a sufficient condition for optimality:
\begin{theorem}\label{thm:converse}
Fix a candidate invariant distribution $\pi \in \cP(\sX)$. Suppose there exist $h_\pi \in L^1(\pi)$, $\lambda_\pi < \infty$, and a stochastic kernel $\Phi^* \in \cK_\pi(R)$ such that
\begin{align}\label{eq:stoch_ACOE_converse}
	\langle \pi, h_\pi \rangle + \lambda_\pi  =  \langle \pi \otimes \Phi^*, c + Qh_\pi \rangle = D_\pi(R; c + Qh_\pi).
\end{align}
Then $\Phi^* \in \cM(\sU|\sX)$ achieves the infimum in \eqref{eq:simplerr2}, and $J^*_\pi(R) = J_{*,\pi}(R) = \lambda_\pi$.
\end{theorem}

\begin{proof} First of all, using the fact that  $\Phi^* \in \cK_\pi$ together with \eqref{eq:stoch_ACOE_converse}, we can write
	\begin{align}\label{eq:lambda_pi}
		\langle \pi \otimes \Phi^*, c \rangle = \langle \pi \otimes \Phi^*, c + Q h_\pi - h_\pi \otimes \1 \rangle = \lambda_\pi
	\end{align}
From Proposition \ref{prop:propaval} and \eqref{eq:stoch_ACOE_converse} we have
	\begin{align*}
	J^*_\pi(R) &= \inf_{\Phi \in \cI_{\pi}(R) } \sup_{h \in L^1(\pi)} \langle \pi \otimes \Phi , c + Qh- h \rangle \\
	&\ge \inf_{\Phi \in \cI_{\pi}(R) } \langle \pi \otimes \Phi , c + Qh_\pi - h_\pi \rangle \\
	&=D_\pi(R; c + Qh_\pi) - \langle \pi, h_\pi \rangle  \\
	&= \lambda_\pi.
	\end{align*}
On the other hand, since $\Phi^* \in \cK_\pi$, we also have
	\begin{align*}
	J^*_\pi(R) &= \inf_{\Phi \in \cI_{\pi}(R) } \sup_{h \in L^1(\pi)} \langle \pi \otimes \Phi , c + Qh - h \rangle \\
	&\le  \sup_{h \in L^1(\pi)} \langle \pi \otimes \Phi^* , c + Qh - h \rangle \\
	&=  \langle \pi \otimes \Phi^*, c \rangle \\
	&=\lambda_\pi,
	\end{align*}
where the last step follows from \eqref{eq:lambda_pi}. This shows that $\langle \pi \otimes \Phi^*, c \rangle = \lambda_\pi = J^*_\pi(R)$, and the optimality of $\Phi^*$ follows.
\end{proof}

To complete the computation of the optimal steady-state value $J^*(R)$ defined in \eqref{eq:steady_state_RI}, we need to consider all candidate invariant distributions $\pi \in \cP(\sX)$ for which $\cK_\pi(R)$ is nonempty, and then choose among them any $\pi$ that attains the smallest value of $J^*_\pi(R)$ (assuming this value is finite). On the other hand, if $J^*_\pi(R) < \infty$ for {\em some} $\pi$, then Theorem~\ref{thm:RI_DRF} ensures that there exists a {\em suboptimal} control law satisfying the information constraint in the steady state.

\subsection{Information-constrained Bellman equation}
\label{ssec:implication}

The function $h_\pi$ that appears in Theorems~\ref{thm:RI_DRF} and \ref{thm:converse} arises 
as a Lagrange multiplier for the invariance constraint $\Phi \in \cK_\pi$.  For a  given invariant measure $\pi \in \cP(\sX)$,
it 
solves the 
fixed-point
equation
\begin{align}\label{eq:IC-BE}
	\langle \pi, h \rangle + \lambda_\pi = \inf_{\Phi \in \cI_\pi(R)} \langle \pi \otimes \Phi, c + Qh \rangle
\end{align}
with $\lambda_\pi = J^*_\pi(R)$. 

In the limit $R \to \infty$ (i.e., as the information constraint is relaxed), while also minimizing over the invariant distribution $\pi$, the optimization problem \eqref{eq:steady_state_RI} reduces to the usual average-cost optimal control problem \eqref{eq:steady_state_RI0}.  
Under appropriate conditions on the model and the cost function, it is known that the solution to \eqref{eq:steady_state_RI0}
is obtained through the associated Average-Cost Optimality Equation (ACOE), or Bellman Equation (BE)
\begin{align}\label{eq:BE}
	h(x) + \lambda = \inf_{u \in \sU} \left[ c(x,u) + Qh(x,u) \right],
\end{align}
with $\lambda = J^*$.  The function $h$ is known as the \textit{relative value function},  and has the same interpretation as a   Lagrange multiplier.

 Based on the similarity between  \eqref{eq:IC-BE} and \eqref{eq:BE}, we 
 refer to the former as the \textit{Information-Constrained Bellman Equation} (or IC-BE). However, while the BE \eqref{eq:BE} gives a fixed-point equation for the relative value function $h$, the existence of a solution pair $(h_\pi,\lambda_\pi)$ 
 for
the IC-BE \eqref{eq:IC-BE} is only a sufficient condition for optimality. By Theorem~\ref{thm:converse}, the Markov kernel $\Phi^*$ that achieves the infimum on the right-hand side of \eqref{eq:IC-BE} must also satisfy the invariance condition $\Phi^* \in \cK_\pi(R)$, which must be verified separately.

In spite of this technicality, the standard BE can be formally recovered in the limit $R \to \infty$.  To demonstrate this, first observe that $J^*_\pi(R)$ is the value of the following (dual) optimization problem:
\begin{align*}
	\text{maximize }  & \lambda \\
	\text{subject to }  & s\left\langle \pi, \log \frac{1}{\int_\sU e^{-\frac{1}{s}[c(\cdot,u)+Qh(\cdot,u)]}\nu(\d u)} - \frac{h}{s}\right\rangle \ge \lambda + s R, \qquad \forall  \nu \in \cP(\sU) \\  & \qquad \quad \lambda \ge 0,\,\,s \ge 0,\,\, h \in L^1(\pi)
\end{align*}
This follows from  \eqref{eq:RI_DRF_dddual}. From the fact that the DRF is convex and nonincreasing in $R$, and from \eqref{eq:tangent},
taking $R\to\infty $ is equivalent to taking $s \to 0$ (with the convention that $s R \to 0$ as $R \to \infty$). Now, {\em Laplace's principle} \cite{DupuisEllis} states that, for any $\nu \in \cP(\sU)$ and any measurable function $F : \sU \to \R$ such that $e^{-F} \in L^1(\nu)$,
\begin{align*}
	-\lim_{s \downarrow 0} s \log \int_{\sU} e^{-\frac{1}{s}F(u)} \nu(\d u) = \nu\text{-}\essinf_{u \in \sU} F(u).
\end{align*}
Thus, the limit of  $J^*_\pi(R)$ as $R \to \infty$ is the value of the optimization problem
\begin{align*}
	\text{maximize } & \lambda \\
	\text{subject to } & \left\langle \pi, \inf_{u \in \sU}\left[c(\cdot,u)+Qh(\cdot,u)\right] - h \right\rangle \ge \lambda, \quad \lambda \ge 0,\,\, h \in L^1(\pi)
\end{align*}
Performing now the minimization over $\pi \in \cP(\sX)$ as well, we see that the limit of $J^*(R)$ as $R \to \infty$ is given by the value of the following problem:
\begin{align*}
	\text{maximize } & \lambda \\
	\text{subject to } & \inf_{u \in \sU} \left[c(\cdot,u)+Qh(\cdot,u)\right] - h \ge \lambda, \quad \lambda \ge 0,\,\, h \in C(\sX)
\end{align*}
which recovers the BE \eqref{eq:BE} (the restriction to continuous $h$ is justified by the fact that continuous functions are dense in $L^1(\pi)$ for any finite Borel measure $\pi$). We emphasize again that this derivation is purely formal, and is intended to illustrate the conceptual relation between the information-constrained control problem and the limiting case as $R \to \infty$.

\subsection{Convergence of mutual information}

So far, we have analyzed the steady-state problem \eqref{eq:steady_state_RI} and provided sufficient conditions for the existence of a pair $(\pi,\Phi^*) \in \cP(\sX) \times \cK$, such that
\begin{align}\label{eq:SS_solution}
	J_\pi(\Phi^*) = J^*_\pi(R) = \inf_{\Phi \in \cK_\pi(R)} \bar{J}_\pi(\Phi) \text{\,\,  $\Pr^{\Phi^*}_\mu$-a.s.} \qquad \text{and} \qquad I(\pi,\Phi^*) = R
\end{align}
(here, $R$ is a given value of the information constraint). Turning to the average-cost problem posed in Section~\ref{sec:formulation}, we can conclude from \eqref{eq:SS_solution} that $\Phi^*$ solves \eqref{eq:RI} in the special case $\mu = \pi$. In fact, in that case the state process $\{X_t\}$ is stationary Markov with $\mu_t = \law(X_t) = \pi$ for all $t$, so we have $I(\mu_t,\Phi^*) = I(\pi,\Phi^*) = R$ for all $t$. However, what if the initial state distribution $\mu$ is different from $\pi$?

For example, suppose that the induced Markov kernel $Q_{\Phi^*} \in \cM(\sX|\sX)$ is weakly ergodic, i.e., $\mu_t$ converges to $\pi$ weakly for any initial state distribution $\mu$. In that case, $\mu_t \otimes \Phi^* \xrightarrow{t \to \infty} \pi \otimes \Phi^*$ weakly as well. Unfortunately, the mutual information functional is only lower semicontinuous in the weak topology, which gives
\begin{align*}
	\liminf_{t \to \infty} I(\mu_t, \Phi^*) \ge I(\pi, \Phi^*) = R.
\end{align*}
That is, while it is reasonably easy to arrange things so that $J_\mu(\Phi^*) = J^*_\pi(R)$ a.s., the information constraint \eqref{eq:info_constraint} will not necessarily be satisfied. The following theorem gives one sufficient condition:

\begin{theorem}\label{thm:info_convergence} Fix a probability measure $\mu \in \cP(\sX)$ and a stable MRS control law $\Phi \in \cM(\sU|\sX)$, and let $\{(X_t,U_t)\}^\infty_{t=1}$ be the corresponding state-action Markov process with $X_1 \sim \mu$. Suppose the following conditions are satisfied:
	\begin{itemize}
		\item[{\rm({\bf I.1})}] The induced transition kernel $Q_\Phi$ is aperiodic and positive Harris recurrent (and thus has a unique invariant probability measure $\pi = \pi Q_\Phi$).
		\item[{\rm({\bf I.2})}] The sequence of information densities
		\begin{align*}
			\imath_t(x,u) \deq \log \frac{\d\, (\mu_t \otimes \Phi)}{\d\, (\mu_t \otimes \mu_t \Phi)}(x,u), \qquad t \ge 1
		\end{align*}
		where $\mu_t = \Pr^\Phi_\mu(X_t \in \cdot)$, is uniformly integrable, i.e.,
		\begin{align}\label{eq:UI}
			\lim_{N \to \infty} \sup_{t \ge 1} \E^\Phi_\mu\left[ \imath_t(X_t,U_t) \1_{\{\imath_t(X_t,U_t) \ge N\}}\right] = 0.
		\end{align}
	\end{itemize}
Then $I(\mu_t,\Phi) \xrightarrow{t \to \infty} I(\pi,\Phi)$.
\end{theorem}
\begin{proof} Since $Q_\Phi$ is  aperiodic and positive Harris recurrent, the sequence $\mu_t$ converges to $\pi$ in total variation (see \cite[Thm.~13.0.1]{meyn_tweedie} or \cite[Thm.~4.3.4]{HernandezLasserreMCbook}):
	\begin{align*}
		\| \mu_t - \pi \|_{{\rm TV}} \deq \sup_{A \in \cB(\sX)} | \mu_t(A) - \pi(A) | \xrightarrow{t \to \infty} 0.
	\end{align*}
By the properties of the total variation distance, $\| \mu_t \otimes \Phi - \pi \otimes \Phi \|_{{\rm TV}} \xrightarrow{t \to \infty} 0$ as well. This, together with the uniform integrability assumption \eqref{eq:UI}, implies that $I(\mu_t,\Phi^*)$ converges to $I(\pi,\Phi^*)$ by a result of Dobrushin \cite{Dobrushin_info_limit}. \hfill \end{proof}

While it is relatively easy to verify the strong ergodicity condition ({\bf I.1}), the uniform integrability requirement ({\bf I.2}) is fairly stringent, and is unlikely to hold except in very special cases:

\begin{example}{\em Suppose that there exist nonnegative $\sigma$-finite measures $\lambda$ on $(\sX,\cB(\sX))$ and $\rho$ on $(\sU,\cB(\sU))$, such that the Radon--Nikodym derivatives
	\begin{align}
		p(x) = \frac{\d \mu}{\d \lambda}(x), \quad f(u|x) = \frac{\d \Phi}{\d \rho}(u|x), \quad g(y|x) = \frac{\d Q_\Phi}{\d \lambda}(y|x)
	\end{align}
exist, and there are constants $c,C > 0$, such that $c \le f(u|x) \le C$ for all $x \in \sX, u \in \sU$. (This boundedness condition will hold only if each of the conditional probability measures $\Phi(\cdot|x), x \in \sX$, is supported on a compact subset $S_x$ of $\sU$, and $\rho(S_x)$ is uniformly bounded.) Then the uniform integrability hypothesis ({\bf I.2}) is fulfilled.

To see this, we first note that, for each $t$, both $\mu_t \otimes \Phi$ and $\mu_t \otimes \mu_t \Phi$ are absolutely continuous w.r.t.\ the product measure $\lambda \otimes \rho$, with
	\begin{align*}
		\frac{\d\, (\mu_t \otimes \Phi)}{\d\, (\lambda \otimes \rho)}(x,u) = p_t(x) f(u|x) \quad \text{and} \quad
		\frac{\d\, (\mu_t \otimes \mu_t \Phi)}{\d\, (\lambda \otimes \rho)}(x,u) = p_t(x) q_t(u),
	\end{align*}
	where $p_1 = p$, and for $t \ge 1$
	\begin{align*}
		p_{t+1}(x) &= \frac{\d \mu_{t+1}}{\d \lambda}(x) = \int_\sX p_{t}(x') g(x|x') \lambda(\d x'), \\
	q_t(u) &= \frac{\d\, (\mu_t \Phi)}{\d \rho}(u) = \int_\sX p_t(x) f(u|x) \lambda(\d x).
	\end{align*}
	This implies that we can express the information densities $\imath_t$ as
	\begin{align*}
		\imath_t(x,u) = \log \frac{f(u|x)}{q_t(u)}, \qquad (x,u) \in \sX \times \sU,\, t = 1,2,\ldots .
	\end{align*}
We then have the following bounds on $\imath_t$:
\begin{align*}
\log \left(\frac{c}{C}\right) \le \imath_t(x,u) &\le \log f(u|x) - \int_\sX p_t(x) \log f(u|x) \lambda(\d x) \le \log \left(\frac{C}{c}\right),
\end{align*}
where in the upper bound we have used Jensen's inequality. Therefore, the sequence of random variables $\{\imath_t(X_t,U_t)\}^\infty_{t=1}$ is uniformly bounded, hence uniformly integrable.}
\end{example}

In certain situations, we can dispense with both the strong ergodicity and the uniform integrability requirements of Theorem~\ref{thm:info_convergence}:

\begin{example}{\em Let $\sX = \sU = \R$. Suppose that the control law $\Phi$ can be realized as a time-invariant linear system
	\begin{align}
		U_t &= k X_t + W_t, \qquad t = 1,2,\ldots
	\end{align}
where $k \in \R$ is the gain, and where $\{W_t\}^\infty_{t=1}$ is a sequence of i.i.d.\ real-valued random variables independent of $X_1$, such that $\nu = \law(W_1)$ has finite mean $m$ and variance $\sigma^2$ and satisfies
\begin{align}
	D(\nu \| N(m,\sigma^2)) < \infty,
\end{align}
where $N(m,\sigma^2)$ denotes a Gaussian probability measure with mean $m$ and variance $\sigma^2$. Suppose also that the induced state transition kernel $Q_\Phi$ with invariant distribution $\pi$ is weakly ergodic, so that $\mu_t \to \pi$ weakly, and additionally that
\begin{align*}
	\lim_{t \to \infty}\int_\sX (x - \langle \mu_t, x \rangle)^2\mu_t(\d x) = \int_\sX (x - \langle \pi, x \rangle)^2 \pi( \d x),
\end{align*}
i.e., the variance of the state converges to its value under the steady-state distribution $\pi$. Then $I(\mu_t,\Phi) \to I(\pi,\Phi)$ as an immediate consequence of Theorem~8 in \cite{wu_verdu_MMSE}.}
\end{example}


\section{Example: information-constrained LQG problem}
\label{sec:IC_LQG}

We now illustrate the general theory presented in the preceding section in the context of an information-constrained version of the well-known Linear Quadratic Gaussian (LQG) control problem. Consider the linear stochastic system
\begin{align}\label{eq:plant}
	X_{t+1} = a X_t + b\,U_t + W_t, \qquad t \ge 1
\end{align}
where $a,b\neq 0$ are the system coefficients, $\{X_t\}^\infty_{t=1}$ is a real-valued state process, $\{U_t\}^\infty_{t=1}$ is a real-valued control process, and $\{W_t\}^\infty_{t=1}$ is a sequence of i.i.d.\ Gaussian random variables with mean $0$ and variance $\sigma^2$. The initial state $X_1$ has some given distribution $\mu$. Here, $\sX = \sU = \R$, and the controlled transition kernel $Q \in \cM(\sX|\sX \times \sU)$ corresponding to \eqref{eq:plant} is $Q(\d y |x,u) = \gamma(y; ax + bu, \sigma^2) \d y$, where $\gamma(\cdot; m,\sigma^2)$ is the probability density of the Gaussian distribution $N(m,\sigma^2)$, and $\d y$ is the Lebesgue measure. We are interested in solving the information-constrained control problem \eqref{eq:RI} with the quadratic cost $c(x,u)  = px^2 + qu^2$ for some given $p,q > 0$.

\begin{theorem}\label{thm:main} Suppose that the system \eqref{eq:plant} is open-loop stable, i.e., $a^2 < 1$. Fix an information constraint $R > 0$. Let $m_1 = m_1(R)$ be the unique positive root of the {\em information-constrained discrete algebraic Riccati equation (IC-DARE)}
\begin{align}\label{eq:IC_DARE}
	p+m(a^2-1)+\frac{(mab)^2}{q+mb^2}(e^{-2R}-1) = 0,
\end{align}
and let $m_2$ be the unique positive root of the standard DARE
	\begin{align}\label{eq:DARE}
		p+m(a^2-1)-\frac{(mab)^2}{q+mb^2} = 0
	\end{align}
Define the control gains $k_1 = k_1(R)$ and $k_2$ by
\begin{align}\label{eq:gain}
	k_i = -\frac{m_iab}{q+m_ib^2}
\end{align}
and steady-state variances $\sigma^2_1 = \sigma^2_1(R)$ and $\sigma^2_2 = \sigma^2_2(R)$ by
\begin{align}
	\sigma^2_i &=  \frac{\sigma^2}{1-\left[e^{-2R}a^2 + (1-e^{-2R})\left(a+bk_i\right)^2\right]}.\label{eq:ss_var}
\end{align}
Then
\begin{align}
\!\!\!\!J^*(R) \le \min\Big( m_1\sigma^2, m_2\sigma^2 + (q+m_2b^2)k^2_2\sigma^2_2e^{-2R}\Big).\label{eq:cost_bound}
\end{align}
Also, let $\Phi_1$ and $\Phi_2$ be two MRS control laws with Gaussian conditional densities
	\begin{align}
	\varphi_i(u|x) &= \frac{\d \Phi_i(u|x)}{\d u} = \gamma\left(u; (1-e^{-2R})k_i x, (1-e^{-2R})e^{-2R}k_i\sigma^2_i\right) ,\label{eq:controller}
	\end{align}
and let $\pi_i = N(0,\sigma^2_i)$ for $i=1,2$. Then the first term on the right-hand side of \eqref{eq:cost_bound} is achieved by $\Phi_1$, the second term is achieved by $\Phi_2$, and $\Phi_i \in \cK_{\pi_i}(R)$ for $i=1,2$. In each case the information constraint is met with equality: $I(\pi_i,\Phi_i)=R$, $i=1,2$.
\end{theorem}

To gain some insight into the conclusions of Theorem~\ref{thm:main}, let us consider some of its implications, and particularly the cases of no information $(R=0)$ and perfect information $(R = +\infty)$. First, when $R=0$, the quadratic IC-DARE \eqref{eq:IC_DARE} reduces to the linear Lyapanov equation \cite{CainesBook} $p +m(a^2-1) = 0$, so the first term on the right-hand side of \eqref{eq:cost_bound} is $m_1(0)\sigma^2 = \frac{p\sigma^2}{1-a^2}$. On the other hand, using Eqs.~\eqref{eq:DARE} and \eqref{eq:gain}, we can show that the second term is equal to the first term, so from \eqref{eq:cost_bound}
\begin{align}\label{eq:no_info_cost_bound}
	J^*(0) \le \frac{p\sigma^2}{1-a^2}.
\end{align} 
Since this is also the minimal average cost in the open-loop case, we have equality in \eqref{eq:no_info_cost_bound}. Also, both controllers $\Phi_1$ and $\Phi_2$ are realized by the deterministic open-loop law $U_t \equiv 0$ for all $t$, as expected. Finally, the steady-state variance is $\sigma^2_1(0) = \sigma^2_2(0) = \frac{\sigma^2}{1-a^2}$, and $\pi_1 =\pi_2 = N(0,\sigma^2/(1-a^2))$, which is the unique invariant distribution of the system \eqref{eq:plant} with zero control  (recall the stability assumption $a^2 < 1$). Second, in the limit $R \to \infty$ the IC-DARE \eqref{eq:IC_DARE} reduces to the usual DARE \eqref{eq:DARE}. Hence, $m_1(\infty) = m_2$, and both terms on the right-hand side of \eqref{eq:cost_bound} are equal to $m_2\sigma^2$:
\begin{align}\label{eq:full_info_cost_bound}
	J^*(\infty) \le m_2\sigma^2.
\end{align}
Since this is the minimal average cost attainable in the scalar LQG control problem with perfect information, we have equality in \eqref{eq:full_info_cost_bound}, as expected. The controllers $\Phi_1$ and $\Phi_2$ are again both deterministic and have the usual linear structure $U_t = k_2 X_t$ for all $t$. The steady-state variance $\sigma^2_1(\infty) =\sigma^2_2(\infty) = \frac{\sigma^2}{1-(a+bk_2)^2}$ is equal to the steady-state variance induced by the optimal controller in the standard (information-unconstrained) LQG problem.

When $0 < R < \infty$, the two control laws $\Phi_1$ and $\Phi_2$ are no longer the same. However, they are both \textit{stochastic} and have the form
\begin{align}\label{eq:controller_AWGN}
	U_t = k_i \left[(1-e^{-2R})X_t + e^{-R}\sqrt{1-e^{-2R}}V^{(i)}_t \right],
\end{align}
where $V^{(i)}_1,V^{(i)}_2,\ldots$ are i.i.d.\ $N(0,\sigma^2_i)$ random variables independent of $\{W_t\}^\infty_{t=1}$ and $X_1$. The corresponding closed-loop system is
\begin{align}\label{eq:CL_system}
	X_{t+1} = \left[a + \left(1-e^{-2R}\right)bk_i \right]X_t + Z^{(i)}_t,
\end{align}
where $Z^{(i)}_1,Z^{(i)}_2,\ldots$ are i.i.d.\ zero-mean Gaussian random variables with variance
\begin{align*}
	\bar{\sigma}^2_i &= e^{-2R}(1-e^{-2R}) \left(bk_i\right)^2 \sigma^2_i + \sigma^2.
\end{align*}
Theorem~\ref{thm:main} implies that, for each $i=1,2$, this system is stable and has the invariant distribution $\pi_i = N(0,\sigma^2_i)$. Moreover, this invariant distribution is unique, and the closed-loop transition kernels $Q_{\Phi_i}$, $i=1,2,$ are ergodic. We also note that the two controllers in \eqref{eq:controller_AWGN} can be realized as a cascade consisting of an additive white Gaussian noise (AWGN) channel and a linear gain:
\begin{align*}
	U_t = k_i\widehat{X}^{(i)}_t, \qquad
	\widehat{X}^{(i)}_t = (1-e^{-2R})X_t + e^{-R}\sqrt{1-e^{-2R}}V^{(i)}_t.
\end{align*}
We can view the stochastic mapping from $X_t$ to $\widehat{X}^{(i)}_t$ as a noisy \textit{sensor} or \textit{state observation channel} that adds just enough noise to the state to satisfy the information constraint in the steady state, while introducing a minimum amount of distortion. The difference between the two control laws $\Phi_1$ and $\Phi_2$ is due to the fact that, for $0 < R < \infty$, $k_1(R) \neq k_2$ and $\sigma^2_1(R) \neq \sigma^2_2(R)$. Note also that the deterministic (linear gain) part of $\Phi_2$ is exactly the same as in the standard LQG problem with perfect information, with or without noise. In particular, the gain $k_2$ is \textit{independent} of the information constraint $R$. Hence, $\Phi_2$ as a \textit{certainty-equivalent} control law which treats the output $\widehat{X}^{(2)}_t$ of the AWGN channel as the best representation of the state $X_t$ given the information constraint. A control law with this structure was proposed by Sims \cite{sims2003implications} on heuristic grounds for the information-constrained LQG problem with discounted cost. On the other hand, for $\Phi_1$ both the noise variance $\sigma^2_1$ in the channel $X_t \to \widehat{X}^{(1)}_t$ and the gain $k_1$ depend on the information constraint $R$. Numerical simulations show that $\Phi_1$ attains smaller steady-state cost for all sufficiently small values of $R$ (see Figure~\ref{fig:cost_comparison}), whereas $\Phi_2$ outperforms $\Phi_1$ when $R$ is large. As shown above, the two controllers are exactly the same (and optimal) in the no-information $(R\to 0)$ and perfect-information $(R\to\infty)$ regimes.

\begin{figure}[htb]
\begin{center}

\includegraphics[width=0.65\columnwidth]{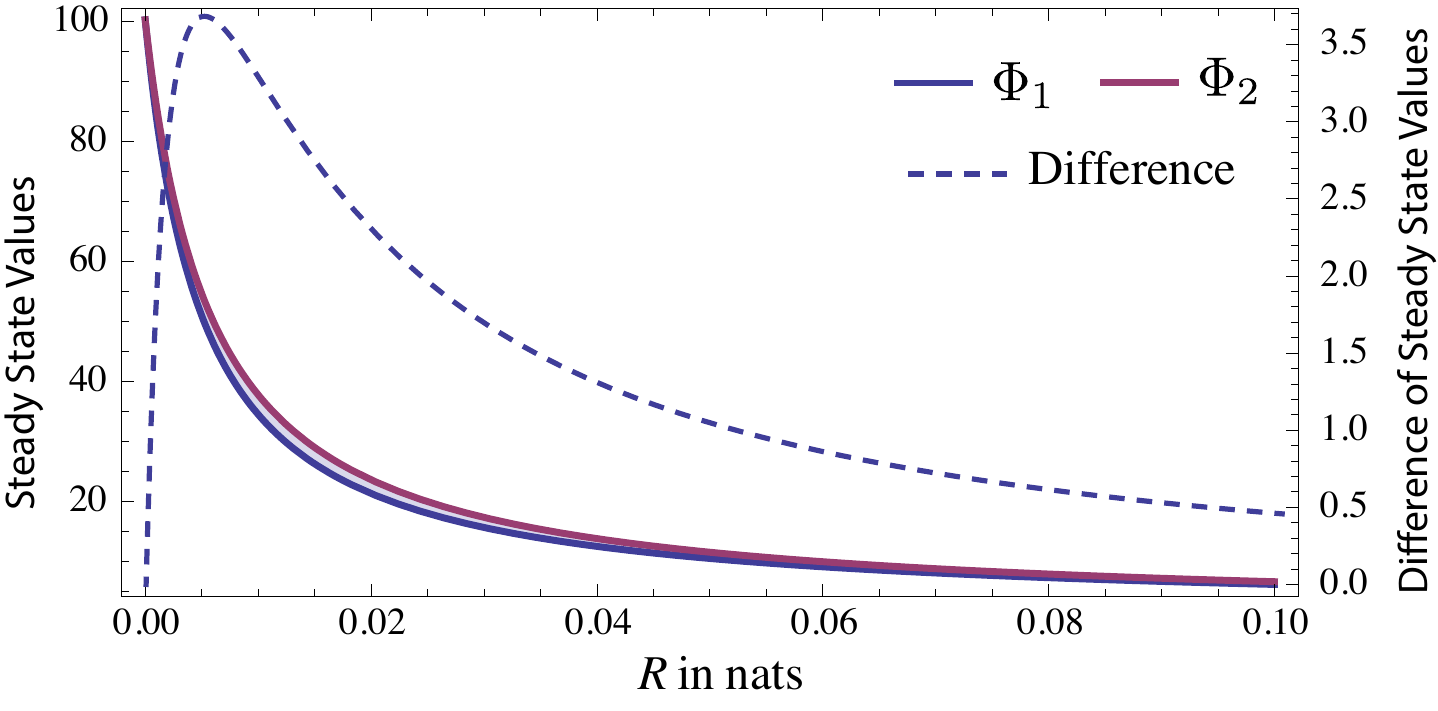}
\end{center}
	\caption{Comparison of $\Phi_1$ and $\Phi_2$ at low information rates and the difference  $\Phi_2-\Phi_1$ (dashed line). System parameters: $a=0.995,b=1,\sigma^2=1$, cost parameters: $p=q=1$.}
	\label{fig:cost_comparison} 
\end{figure}

In the unstable case $(a^2 >1)$, a simple sufficient condition for the existence of an information-constrained controller that results in a stable closed-loop system is
\begin{align}\label{eq:rate_condition}
	R > \frac{1}{2}\log \frac{a^2 -(a+bk_2)^2}{1-(a+bk_2)^2},
\end{align}
where  $k_2$ is given by \eqref{eq:gain}. Indeed, if $R$ satisfies \eqref{eq:rate_condition}, then the steady-state variance $\sigma^2_2$ is well-defined, so the closed-loop system \eqref{eq:CL_system} with $i=2$ is stable.

\subsection{Proof of Theorem~\ref{thm:main}}

We will show that the pairs $(h_i,\lambda_i)$ with
\begin{align*}
	&h_1(x) = m_1 x^2, \quad \lambda_1 = m_1 \sigma^2 \\
	&h_2(x) = m_2 x^2, \quad \lambda_2 = m_2 \sigma^2 + (q+m_2b^2)k^2_2\sigma^2_2e^{-2R}
\end{align*}
both solve the IC-BE \eqref{eq:IC-BE} for $\pi_i$, i.e.,
\begin{align}\label{eq:stochastic_ACOE_2}
	\langle \pi_i, h_i \rangle + \lambda_i = D_{\pi_i}(R; c + Qh_i),
\end{align}
and that the MRS control law $\Phi_i$ achieves the value of the distortion-rate function in \eqref{eq:stochastic_ACOE_2} and belongs to the set $\cK_{\pi_i}(R)$. Then the desired results will follow from Theorem~\ref{thm:converse}. We split the proof into several logical steps.

\paragraph{Step 1: Existence, uniqueness, and closed-loop stability}

We first demonstrate that $m_1 = m_1(R)$ indeed exists and is positive, and that the steady-state variances $\sigma^2_1$ and $\sigma^2_2$ are finite and positive. This will imply that the closed-loop system \eqref{eq:CL_system} is stable and ergodic with the unique invariant distribution $\pi_i$. (Uniqueness and positivity of $m_2$ follow from well-known results on the standard LQG problem.)

\begin{lemma}\label{lm:existence} For all $a,b \neq 0$ and all $p,q,R > 0$, Eq.~\eqref{eq:IC_DARE} has a unique positive root $m_1=m_1(R)$.
\end{lemma}

\begin{proof} It is a straightforward exercise in calculus to prove that the function
	\begin{align*}
		F(m) \deq p + ma^2 + \frac{(mab)^2}{q+mb^2}(e^{-2R}-1).
	\end{align*}
is strictly increasing and concave for $m > -q/b^2$. Therefore, the fixed-point equation $F(m)=m$ has a unique positive root $m_1(R)$. (See the proof of Proposition~4.1 in \cite{Bertsekas} for a similar argument.) \hfill\end{proof}

\begin{lemma}\label{lm:CL_stability} For all $a,b \neq 0$ with $a^2 < 1$ and $p,q,R > 0$,
	\begin{align}\label{eq:CL_stability}
		e^{-2R}a^2 + (1-e^{-2R})(a+bk_i)^2 \in (0,1), \quad i=1,2.
	\end{align}
Thus, the steady-state variance $\sigma^2_i = \sigma^2_i(R)$ defined in \eqref{eq:ss_var} is finite and positive.
\end{lemma}
\begin{proof} We write
	\begin{align*}
		 e^{-2R}a^2 + (1-e^{-2R})(a+bk_i)^2 = e^{-2R}a^2 + (1-e^{-2R})\left[a\left(1 - \frac{m_i b^2}{q+m_i b^2}\right)\right]^2 \le a^2,
	\end{align*}
where the second step uses \eqref{eq:gain} and the last step follows from the fact that $q > 0$ and $m_i > 0$ (cf.~Lemma~\ref{lm:existence}). We get \eqref{eq:CL_stability} from open-loop stability ($a^2 < 1$).
\end{proof}

\paragraph{Step 2: A quadratic ansatz for the relative value function} Let $h(x) = mx^2$ for an arbitrary $m > 0$. Then
\begin{equation}
	Qh(x,u) = \int_\sX h(y)Q(\d y|x,u) = m(ax+bu)^2 + m\sigma^2,
\end{equation}
and
\begin{equation}
	c(x,u) + Qh(x,u) \nonumber= m\sigma^2+ (q+mb^2)\left(u - \tilde{x}\right)^2 + \left(p+ma^2 - \frac{m^2(ab)^2}{q+mb^2}\right)x^2,
\end{equation}
where we have set $\tilde{x} = -\dfrac{mab}{q+mb^2}x$. Therefore, for any $\pi \in \cP(\sX)$ and any $\Phi \in \cM(\sU|\sX)$, such that $\pi$ and $\pi\Phi$ have finite second moments, we have
\begin{align*}
 \langle \pi \otimes \Phi, c + Qh - h \rangle
	&= m\sigma^2 + \left(p+m(a^2 - 1) -\frac{(mab)^2}{q+mb^2}\right)\int_\sX x^2 \pi(\d x) \nonumber\\
	& \qquad \qquad + (q+mb^2) \int_{\sX \times \sU} (u-\tilde{x})^2\pi(\d x) \Phi(\d u|x).
\end{align*} 
\paragraph{Step 3: Reduction to a static Gaussian rate-distortion problem} Now we consider the Gaussian case $\pi = N(0,\upsilon)$ with an arbitrary $\upsilon > 0$. Then for any $\Phi \in \cM(\sU|\sX)$
\begin{align*}
&	\langle \pi \otimes \Phi, c + Qh - h \rangle \\
& = m\sigma^2 + \left(p+m(a^2 - 1) -\frac{(mab)^2}{q+mb^2}\right)\upsilon + (q+mb^2) \int_{\sX \times \sU} (u-\tilde{x})^2\pi(\d x) \Phi(\d u|x).
\end{align*}
We need to minimize the above over all $\Phi \in \cI_\pi(R)$. If $X$ is a random variable with distribution $\pi = N(0,\upsilon)$, then its scaled version
\begin{align}\tilde{X} = -\displaystyle\frac{mab}{q+mb^2}X \equiv kX\label{eq:scaled_state}
\end{align}
has distribution $\tilde{\pi} = N(0,\tilde{\upsilon})$ with $\tilde{\upsilon} = k^2\upsilon$. Since the transformation $X \mapsto \tilde{X}$ is one-to-one and the mutual information is invariant under one-to-one transformations \cite{PinskerBook}, 
\begin{align}
	D_\pi(R; c + Qh) - \langle \pi, h \rangle 
&	=\inf_{\Phi \in \cI_\pi(R)} \langle \pi \otimes \Phi, c + Qh - h \rangle \label{eq:J_pi_Gauss_0} \\
& =  m\sigma^2 +  \left(p+m(a^2 - 1) -\frac{(mab)^2}{q+mb^2}\right)\upsilon\nonumber\\
& \qquad  +(q+mb^2) \inf_{\tilde{\Phi} \in \cI_{\tilde{\pi}}(R)}\int_{\sX \times \sU} (u-\tilde{x})^2\tilde{\pi}(\d \tilde{x}) \tilde{\Phi}(\d u|\tilde{x}). \label{eq:J_pi_Gauss}
\end{align}
We recognize the infimum in \eqref{eq:J_pi_Gauss} as the DRF for the Gaussian distribution $\tilde{\pi}$ w.r.t.\ the squared-error distortion $d(\tilde{x},u) = (\tilde{x}-u)^2$. (See Appendix \ref{ssec:gaussian} for a summary of  standard results on the Gaussian DRF.) Hence,
\begin{align}
& D_\pi(R; c + Qh) - \langle \pi, h \rangle \nonumber\\
&=  m\sigma^2  +  \left(p+m(a^2 - 1) -\frac{(mab)^2}{q+mb^2}\right)\upsilon  +(q+mb^2)\tilde{\upsilon} e^{-2R} \nonumber\\
&=  m\sigma^2  +  \left(p+m(a^2 - 1) +\frac{(mab)^2}{q+mb^2}(e^{-2R}-1)\right)\upsilon\label{eq:Bellman_DRF_1}\\
&=  m\sigma^2  +  \left(p+m(a^2 - 1) -\frac{(mab)^2}{q+mb^2}\right)\upsilon +(q+mb^2)k^2\upsilon e^{-2R} \label{eq:Bellman_DRF_2},
\end{align}
where Eqs.~\eqref{eq:Bellman_DRF_1} and \eqref{eq:Bellman_DRF_2} are obtained by collecting appropriate terms and using the definition of $k$ from \eqref{eq:scaled_state}.
We can now state the following result:
\begin{lemma}\label{lm:IC_ACOE} Let $\pi_i = N(0,\sigma^2_i)$, $i=1,2$. Then the pair $(h_i,\lambda_i)$ 
solves the information-constrained ACOE \eqref{eq:stochastic_ACOE_2}. Moreover, for each $i$ the controller $\Phi_i$ defined in \eqref{eq:controller} achieves the DRF in \eqref{eq:stochastic_ACOE_2} and belongs to the set $\cK_{\pi_i}(R)$.
\end{lemma} 

\begin{proof} If we let $m=m_1$, then the second term in \eqref{eq:Bellman_DRF_1} is identically zero for any $\upsilon$. Similarly, if we let $m=m_2$, then the second term in \eqref{eq:Bellman_DRF_2} is zero for any $\upsilon$. In each case, the choice $\upsilon = \sigma^2_i$ gives \eqref{eq:stochastic_ACOE_2}. From the results on the Gaussian DRF (see Appendix \ref{ssec:gaussian}), we know that, for a given $\upsilon > 0$, the infimum in \eqref{eq:J_pi_Gauss} is achieved by 
\begin{align*}
&	K^*_i(\d u |\tilde{x}) = \gamma\left( u; (1-e^{-2R})\tilde{x}, e^{-2R}(1-e^{-2R})\tilde{\upsilon}\right) \d u.
\end{align*}
Setting $\upsilon = \sigma^2_i$ for $i=1,2$ and using $\tilde{x}=k_ix$ and $\tilde{\upsilon} = k_i^2\sigma^2_i$, we see that the infimum over $\Phi$ in \eqref{eq:J_pi_Gauss_0} in each case is achieved by composing the deterministic mapping
\begin{align}\label{eq:scaling}
\tilde{x} = k_i x = -\frac{m_i ab}{q+m_i b^2}x
\end{align}
with $K^*_i$. It is easy to see that this composition is precisely the stochastic control law $\Phi_i$ defined in \eqref{eq:controller}. Since the map \eqref{eq:scaling} is one-to-one, we have $I(\pi_i, \Phi_i) = I(\tilde{\pi}_i, K^*_i) = R$. Therefore, $\Phi_i \in \cI_{\pi_i}(R)$.

It remains to show that $\Phi_i \in \cK_{\pi_i}$, i.e., that $\pi_i$ is an invariant distribution of $Q_{\Phi_i}$. This follows immediately from the fact that $Q_{\Phi_i}$ is realized as
\begin{align*}
	Y = (a+bk_i e^{-2R})X + bk_i e^{-R}\sqrt{1-e^{-2R}}V^{(i)} + W,
\end{align*}
where $V^{(i)} \sim N(0,\sigma^2_i)$ and $W \sim N(0,\sigma^2)$ are independent of one another and of $X$ [cf.~\eqref{eq:Gaussian_DRF_FC}]. If $X \sim \pi_i$, then the variance of the output $Y$ is equal to
\begin{align*}
&	(a + bk_i e^{-2R})^2 \sigma^2_i + (bk_i)^2 e^{-2R}(1-e^{-2R})\sigma^2_i + \sigma^2 \\ 
&\qquad = \left[e^{-2R}a^2 + (1-e^{-2R})\left(a + bk_i\right)^2\right] \sigma^2_i + \sigma^2 = \sigma^2_i,
\end{align*}
where the last step follows from \eqref{eq:ss_var}. This completes the proof of the lemma.
\end{proof}
Putting together Lemmas~\ref{lm:existence}--\ref{lm:IC_ACOE} and using Theorem~\ref{thm:converse}, we obtain Theorem~\ref{thm:main}.


\section{Infinite-horizon discounted-cost problem}
\label{sec:discounted}

We now consider the problem of rationally inattentive control subject to the infinite-horizon discounted-cost criterion. This is the setting originally considered by Sims \cite{sims2003implications, sims2006rational}.  The approach followed in that work was to select, for each time $t$, an observation channel that would provide the best estimate $\widehat{X}_t$ of the state $X_t$ under the information constraint, and then invoke the principle of certainty equivalence to pick a control law that would map the estimated state to the control $U_t$, such that the joint process $\{(X_t,\widehat{X}_t,U_t)\}$ would be stationary. On the other hand, the discounted-cost criterion by its very nature places more emphasis on the transient behavior of the controlled process, since the costs incurred at initial stages contribute the most to the overall expected cost. Thus, even though the optimal control law may be stationary, the state \textit{process} will not be. With this in mind, we propose an alternative methodology that builds on the convex-analytic approach and results in control laws that perform well not only in the long term, but also in the transient regime.

In this section only, for ease of bookkeeping, we will start the time index at $t=0$ instead of $t=1$. As before, we consider a controlled Markov chain with transition kernel $Q \in \cM(\sX|\sX,\sU)$ and initial state distribution $\mu \in \cP(\sX)$ of $X_0$.   However, we now allow \textit{time-varying} control strategies and refer to any sequence $\bd{\Phi} = \{\Phi_t\}^\infty_{t=0}$ of Markov kernels $\Phi_t \in \cM(\sU|\sX)$ as a \textit{Markov randomized} (MR) control law. We let $\Pr^{\bd{\Phi}}_\mu$ denote the resulting process distribution of $\{(X_t,U_t)\}^\infty_{t=0}$, with the corresponding expectation denoted by $\E^{\bd{\Phi}}_\mu$. Given a measurable one-step state-action cost $c : \sX \times \sU \to \R$ and a discount factor $0 < \beta < 1$, we can now define the infinite-horizon discounted cost as
\begin{align*}
	J^\beta_\mu(\bd{\Phi}) \deq \E^{\bd{\Phi}}_\mu \left[ \sum^\infty_{t=0} \beta^{t} c(X_t,U_t)\right].
\end{align*}
Any MRS control law $\Phi \in \cM(\sU|\sX)$ corresponds to having $\Phi_t = \Phi$ for all $t$, and in that case we will abuse the notation a bit and write $\Pr^\Phi_\mu$, $\E^\Phi_\mu$, and $J^\beta_\mu(\Phi)$. In addition, we say that a control law $\bd{\Phi}$ is \textit{Markov randomized quasistationary} (MRQ) if there exist two Markov kernels $\Phi^{(0)},\Phi^{(1)} \in \cM(\sU|\sX)$ and a deterministic time $t_0 \in {\mathbb Z}_+$, such that $\Phi_t$ is equal to $\Phi^{(0)}$ for $t < t_0$ and $\Phi^{(1)}$ for $t \ge t_0$.

We can now formulate the following information-constrained control problem:
\begin{subequations} \label{eq:infproblem}
\begin{align}
\text{minimize } & J^\beta_\mu(\bd{\Phi}) \label{eq:infprob1} \\
\text{subject to } & I(\mu_t, \Phi_t) \le R, \,\, \forall t \ge 0 . \label{eq:infprob2}
\end{align}
\end{subequations}
Here, as before, $\mu_t = \Pr^{\bd{\Phi}}_\mu[X_t \in \cdot]$ is the distribution of the state at time $t$, and the minimization is over all MRQ control laws $\bd{\Phi}$.

\subsection{Reduction to single-stage optimization}
\label{ssec:infreduction2}

In order to follow the convex-analytic approach as in Section~\ref{ssec:reduction1}, we need to write \eqref{eq:infproblem} as an expected value of the cost $c$ with respect to an appropriately defined probability measure on $\sX \times \sU$. In contrast to what we had for \eqref{eq:RI}, the optimal solution here will depend on the initial state distribution $\mu$. We impose the following assumptions:
\begin{itemize}
	\item[({\bf D.1})] The state space $\sX$ and the action space $\sU$ are compact.
	\item[({\bf D.2})] The transition kernel $Q$ is weakly continuous.
	\item[({\bf D.3})] The cost function $c$ is nonnegative, lower semicontinuous, and bounded.
\end{itemize} 
The essence of the convex-analytic approach to infinite-horizon discounted-cost optimal control is in the following result \cite{Borkar_convex_analytic}:
\begin{proposition} 
\label{prop:occupation} For any MRS control law $\Phi \in \cM(\sU|\sX)$, we have
\begin{align*}
	J^\beta_\mu(\Phi) = \frac{1}{1-\beta} \big\langle \Gamma^\beta_{\mu,\Phi},c \big\rangle,
\end{align*}
where $\Gamma^\beta_{\mu,\Phi} \in \cP(\sX \times \sU)$ is the {\em discounted occupation measure}, defined by
\begin{align}\label{eq:discounted_measure}
	\big\langle \Gamma^\beta_{\mu,\Phi}, f \big\rangle = (1-\beta) \E^{\Phi}_\mu \left[ \sum^\infty_{t=0} \beta^t f(X_t,U_t)\right], \qquad \forall f \in C_b(\sX \times \sU).
\end{align}
This measure can be disintegrated as $\Gamma^\beta_{\mu,\Phi} = \pi \otimes \Phi$, where $\pi \in \cP(\sX)$ is the unique solution of the equation
\begin{align}\label{eq:disc_invariant}
	\pi = (1-\beta)\mu + \beta \pi Q_\Phi.
\end{align}
\end{proposition}It is well-known that, in the absence of information constraints, the minimum of $J^\beta_\mu(\bd{\Phi})$  is achieved by an MRS policy. Thus, if we define the set
\begin{align*}
	{\cal G}^\beta_\mu \deq \Big\{ \Gamma = \pi \otimes \Phi \in \cP(\sX \times \sU) : \pi = (1-\beta)\mu + \beta \pi Q_\Phi \Big\},
\end{align*}
then, by Proposition~\ref{prop:occupation},
\begin{align}\label{eq:min_discounted}
	J^{\beta*}_\mu \deq \inf_{\bd{\Phi}} J^\beta_\mu(\bd{\Phi}) = \frac{1}{1-\beta}\inf_{\Gamma \in {\cal G}^\beta_\mu} \langle \Gamma, c\rangle,
\end{align}
and if $\Gamma^* = \pi^* \otimes \Phi^*$ achieves the infimum, then $\Phi^*$ gives the optimal MRS control law. We will also need the following approximation result:

\begin{proposition}\label{prop:info_approx} For any MRS control law $\Phi \in \cM(\sU|\sX)$ and any $\eps > 0$, there exists an MRQ control law $\bd{\Phi}^\eps$, such that
	\begin{align}\label{eq:cost_approx}
		J^\beta_\mu(\bd{\Phi}^\eps) \le J^\beta_\mu(\Phi) + \eps,
	\end{align}
	and
	\begin{align}\label{eq:info_approx}
		I(\mu^\eps_t,\Phi^\eps_t) \le \frac{C}{(1-\beta)^2\eps} I(\pi, \Phi), \qquad t = 0,1,\ldots
	\end{align}
	where $\mu^\eps_t = \Pr^{\bd{\Phi}^\eps}_\mu(X_t \in \cdot)$, $\pi \in \cP(\sX)$ is given by \eqref{eq:disc_invariant}, and $C = \displaystyle{\max_{x \in \sX}\max_{u \in \sU}c(x,u)}$.
\end{proposition}
\begin{proof} Given an MRS $\Phi$, we construct $\bd{\Phi}^\eps$ as follows:
\begin{align*}
	\Phi^\eps_t(\d u |x) &= \begin{cases}
	\Phi(\d u |x), & t < t_* \\
	\delta_{u_0}(\d u), & t \ge t_*
\end{cases},
\end{align*}
where
\begin{align}\label{eq:cutoff}
	t_* \deq \min \left\{ t \in {\mathbb N} : \frac{C \beta^t}{1-\beta} \le \eps\right\},
\end{align}
and $u_0$ is an arbitrary point in $\sU$. For each $t$, let $\mu_t = \mu Q^t_\Phi = \Pr^\Phi_\mu(X_t \in \cdot)$. Then, using the Markov property and the definition \eqref{eq:cutoff} of $t_*$, we have
\begin{align*}
	J^\beta_\mu(\bd{\Phi}^\eps) &= \E^{\Phi}_\mu\left[ \sum^{t_*-1}_{t=0} \beta^t c(X_t,U_t)\right] + \beta^{t_*}\E^{\delta_{u_0}}_{\mu_{t_*}}\left[\sum^\infty_{t = 0} \beta^t c(X_t,u_0)\right] \\
	&\le J^\beta_\mu(\Phi) + C\beta^{t_*}\sum^\infty_{t=0}\beta^t \\
	&\le J^\beta_\mu(\Phi) + \eps,
\end{align*}
which proves \eqref{eq:cost_approx}. To prove \eqref{eq:info_approx}, we note that \eqref{eq:discounted_measure} implies that
\begin{align*}
	\Gamma^\beta_{\mu,\Phi} = \pi \otimes \Phi = \left((1-\beta)\sum^\infty_{t=0} \beta^t \mu Q^t_\Phi \right) \otimes \Phi.
\end{align*}
Therefore, since the mutual information $I(\nu,K)$ is concave in $\nu$, we have
\begin{align*}
	I(\pi,\Phi) &\ge (1-\beta)\sum^\infty_{t=0} \beta^t I(\mu Q^t_\Phi, \Phi) \\
	&= (1-\beta)\sum^{t_*-1}_{t=0} \beta^t I(\mu^\eps_t,\Phi^\eps_t) + (1-\beta)\sum^\infty_{t=t_*} \beta^t I(\mu_t,\Phi_t) \\
	&\ge (1-\beta)\sum^{t_*-1}_{t=0} \beta^t I(\mu^\eps_t,\Phi^\eps_t) \\
	&\ge (1-\beta)\beta^{{t_*-1}}\max_{0 \le t < t_*} I(\mu^\eps_t,\Phi^\eps_t),
\end{align*}
where we have also used the fact that the mutual information is nonnegative, as well as the definition of $t_*$. This implies that, for $t < t_*$,
\begin{align*}
	I(\mu^\eps_t,\Phi^\eps_t) \le \frac{I(\pi,\Phi)}{(1-\beta)\beta^{t_*-1}} \le \frac{C}{(1-\beta)^2\eps} I(\pi,\Phi).
\end{align*}
For $t \ge t_*$, $I(\mu^\eps_t,\Phi^\eps_t) = 0$, since at those time steps the control $U_t$ is independent of the state $X_t$ by construction of $\bd{\Phi}^\eps$.
\end{proof}

As a consequence of Propositions~\ref{prop:occupation} and \ref{prop:info_approx}, we can now focus on the following static information-constrained problem:
\begin{subequations} \label{eq:infinfconv}
\begin{align}
\text{minimize } & \frac{1}{1-\beta}\langle \Gamma , c \rangle \label{eq:infinfconv1} \\
\text{subject to } & \Gamma \in {\cal G}^\beta_\mu, \, I(\Gamma) \le \bar{R} \label{eq:infinfconv2}
\end{align}
\end{subequations}
(the information constraint $\bar{R}$ will be related to the original value $R$ later). We will denote the value of this optimization problem by $J^{\beta*}_\mu(\bar{R})$.

\subsection{Marginal decomposition}
\label{ssec:bellmaninf}

We now follow more or less the same route as we did in Section~\ref{ssec:bellman} for the average-cost case. Given $\pi \in \cP(\sX)$, let us define the set
\begin{align*}
	\cK_{\mu, \pi }^{\beta} \deq  \Big\{ \Phi \in \cM(\sU|\sX): \pi = (1 - \beta) \mu + \beta \pi Q_{\Phi} \Big\}
\end{align*}
(this set may very well be empty, but, for example, $\cK^\beta_{\mu,\mu} = \cK_\mu$). We can then decompose the infimum in \eqref{eq:min_discounted} as
\begin{align}\label{eq:infsteady_state_RI0}
	J^{\beta*}_{\mu} = \frac{1}{1-\beta}\inf_{\Gamma \in {\cal G}^\beta_\mu} \langle \Gamma, c \rangle = \frac{1}{1-\beta}\inf_{\pi \in \cP(\sX)} \inf_{\Phi \in \cK_{\mu , \pi}^\beta} \langle \pi \otimes \Phi, c \rangle.
\end{align}
If we further define $\cK^\beta_{\mu,\pi}(\bar{R}) \deq \cK^\beta_{\mu,\pi} \cap \cI_\pi(\bar{R})$, then the value of the optimization problem \eqref{eq:infinfconv} will be given by
\begin{align}\label{eq:infsteady_state_RI}
J^{\beta*}_{\mu}(\bar{R}) = \inf_{\pi \in \cP(\sX)} J^{\beta*}_{\mu,\pi}(\bar{R}), \qquad \text{where } J^{\beta*}_{\mu,\pi}(R) \deq \frac{1}{1-\beta} \inf_{\Phi \in \cK^\beta_{\mu,\pi}(\bar{R})} \langle \pi \otimes \Phi, c \rangle .
\end{align}
From here onward, the progress is very similar to what we had in Section~\ref{ssec:bellman}, so we omit the proofs for the sake of brevity. We first decouple the condition $\Phi \in \cK^\beta_{\mu,\pi}$ from the information constraint $\Phi \in \cI_\pi(\bar{R})$ by introducing a Lagrange multiplier:

\begin{proposition}\label{prop:infpropaval} For any $\pi \in \cP(\sX)$,
\begin{align}\label{eqinf:J_pi}
J^{\beta*}_{\mu , \pi}(R) = \frac{1}{1-\beta}\inf_{\Phi \in \cI_\pi(\bar{R})} \sup_{h \in C_b(\sX)}  \left[ \langle \pi \otimes \Phi , c + \beta Qh  - h \otimes \1 \rangle +(1 - \beta)  \langle  \mu  , h\rangle \right].
\end{align}
\end{proposition}Since the cost $c$ bounded, $J^{\beta*}_{\mu,\pi}(\bar{R}) < \infty$, and we may interchange the order of the infimum and the supremum with the same justification as in the average-cost case:
\begin{align}\label{eq:infRI_dual}
J^{\beta*}_{\mu,\pi}(\bar{R}) = \frac{1}{1-\beta}\sup_{h \in C_b(\sX)} \inf_{\Phi \in \cI_\pi(\bar{R})} \left[ \langle \pi \otimes \Phi , c + \beta Qh - h \otimes \1 \rangle +(1 - \beta)  \langle  \mu  , h \rangle \right]
\end{align}
At this point, we have reduced our problem to the form that can be handled using rate-distortion theory:

\begin{theorem}\label{thm:infRI_DRF} Consider a probability measure $\pi \in \cP(\sX)$, and suppose that the supremum over $h \in C_b(\sX)$ in \eqref{eq:RI_dual} is attained by some $h^\beta_{\mu,\pi}$. Then there exists an MRS control law $\Phi^* \in \cM(\sU|\sX)$ such that $I(\pi,\Phi^*) = \bar{R}$, and we have
	\begin{align} \label{infavalacoe}
		& J^{\beta*}_{\mu,\pi}(\bar{R}) + \frac{1}{1-\beta}\langle \pi , h^\beta_{\mu,\pi} \rangle -  \langle \mu, h^\beta_{\mu,\pi}  \rangle \nonumber\\
		& \qquad \qquad = \frac{1}{1-\beta}\langle \pi \otimes \Phi^*, c + \beta Q h^\beta_{\mu,\pi}\rangle \nonumber\\
		& \qquad \qquad = \frac{1}{1-\beta} D_{\pi}(\bar{R};c + \beta Qh^\beta_{\mu,\pi})  .
	\end{align}
Conversely, if there exist a function $h^\beta_{\mu,\pi} \in L^1(\pi)$, a constant $\lambda^\beta_{\mu,\pi} > 0$, and a Markov kernel $\Phi^* \in \cK^\beta_{\mu,\pi}(\bar{R})$, such that
\begin{align} \label{eq:IC-DCOE}
&	 \frac{1}{1-\beta}\langle \pi , h^\beta_{\mu,\pi} \rangle -  \langle \mu, h^\beta_{\mu,\pi}  \rangle + \lambda^\beta_{\mu,\pi} \nonumber\\
	& \qquad \qquad = \frac{1}{1-\beta}\langle \pi \otimes \Phi^*, c + \beta Q h^\beta_{\mu,\pi}\rangle  \nonumber\\
	& \qquad \qquad = \frac{1}{1-\beta} D_{\pi}(\bar{R};c + \beta Qh^\beta_{\mu,\pi}),
\end{align}
then $J^{\beta*}_{\mu,\pi}(\bar{R}) = \lambda^\beta_{\mu,\pi}$, and this value is achieved by $\Gamma^* = \pi \otimes \Phi^*$.
\end{theorem}

The gist of Theorem~\ref{thm:infRI_DRF} is that the original dynamic control problem is reduced to a static rate-distortion problem, where the distortion function is obtained by perturbing the one-step cost $c(x,u)$ by the discounted value of the state-action pair $(x,u)$.

\begin{theorem} Given $R \ge 0$ and $\eps > 0$, suppose that Eq.~\eqref{eq:IC-DCOE} admits a solution triple $(h^\beta_{\mu,\pi},\lambda^\beta_{\mu,\pi},\Phi^*)$ with
	$$
\bar{R} \equiv \bar{R}(\eps,\beta) \deq \frac{(1-\beta)^2\eps}{C}R.
$$
Let ${\cal Q}_\mu(R)$ denote the set of all MRQ control laws $\bd{\Phi}$ satisfying the information constraint \eqref{eq:infprob2}. Then
	\begin{align}\label{eq:discounted_DRF_MRQ}
		\inf_{\bd{\Phi} \in {\cal Q}_\mu(R) } J^\beta_\mu(\bd{\Phi}) \le \frac{1}{1-\beta} \left[D_{\pi}\big(\bar{R}(\eps,\beta);c + \beta Qh^\beta_{\mu,\pi}\big) - \langle \pi, h^\beta_{\mu,\pi}\rangle\right] + \langle \mu, h^\beta_{\mu,\pi}\rangle + \eps.
	\end{align}
\end{theorem}
\begin{proof} Given $\Phi^*$ and $\eps > 0$, Proposition~\ref{prop:info_approx} guarantees the existence of a MRQ control strategy $\bd{\Phi}^{\eps *}$, such that
	\begin{align*}
		J^\beta_\mu(\bd{\Phi}^{\eps*}) \le J^\beta_\mu(\Phi^*) + \eps = \lambda^\beta_{\mu,\pi} + \eps
	\end{align*}
	and $I(\mu_t,\Phi^{\eps *}_t) \le R$ for all $t \ge 0$. Thus, $\bd{\Phi}^{\eps*} \in {\cal Q}_\mu(R)$. Taking the infimum over all $\bd{\Phi} \in {\cal Q}_\mu(R)$ and using \eqref{eq:IC-DCOE}, we obtain \eqref{eq:discounted_DRF_MRQ}.
\end{proof}




\begin{appendix}
	
	\bigskip

\centerline{\large\bf Appendices}

	\section{Sufficiency of memoryless observation channels}
	\label{ssec:standardproblem}	
	
In Sec.~\ref{sec:formulation}, we have focused our attention to information-constrained control problems, in which the control action $U_t$ at each time $t$ is determined only on the basis of the (noisy) observation $Z_t$ pertaining to the current state $X_t$. We also claimed that this restriction to \textit{memoryless} observation channels entails no loss of generality, provided the control action at time $t$ is based only on $Z_t$ (i.e., the information structure is amnesic in the terminology of \cite{WitEqu} --- the controller is forced to ``forget'' $Z_1,\ldots,Z_{t-1}$ by time $t$). In this Appendix, we provide a rigorous justification of this claim for a class of models that subsumes the set-up of Section~\ref{sec:formulation}. One should keep in mind, however, that this claim is unlikely to be valid when the controller has access to $Z^t$. 

We consider the same model as in Section~\ref{sec:formulation}, except that we replace the model components ({\bf M.3}) and ({\bf M.4}) with
\begin{itemize}
\item[({\bf M.3'})] the observation channel, specified by a sequence $\bd{W}$ of stochastic kernels $W_t \in \cM(\sZ | \sX^t \times \sZ^{t-1} \times \sU^{t-1})$, $t = 1,2,\ldots$;
\item[({\bf M.4'})] the feedback controller, specified by a sequence $\bd{\Phi}$ of stochastic kernels $\Phi_t \in \cM(\sU | \sZ)$, $t=1,2,\ldots$.
\end{itemize}
We also consider a \textit{finite-horizon} variant of the control problem \eqref{eq:RI}. Thus, the DM's problem is to {\em design} a suitable channel $\bd{W}$ and a controller $\bd{\Phi}$ to minimize the expected total cost over $T < \infty$ time steps subject to an information constraint:
\begin{subequations}\label{eq:baby_RI}
\begin{align}
\text{minimize }	& \E^{\bd{\Phi},\bd{W}}_\mu\left[ \sum^T_{t=1} c(X_t,U_t)\right] \label{eq:baby_RI_AC}\\
\text{subject to } & I(X_t; Z_t) \le R, \,\, t = 1,2,\ldots,T \label{eq:baby_RI_info}
\end{align}
\end{subequations}
The optimization problem \eqref{eq:baby_RI} seems formidable: for each time step $t = 1,\ldots,T$ we must design stochastic kernels $W_t(\d z_t|x^t,z^{t-1},u^{t-1})$ and $\Phi_t(\d u_t|z_t)$ for the observation channel and the controller, and the complexity of the feasible set of $W_t$'s grows with $t$. However, the fact that (a) both the controlled system and the controller are Markov, and (b) the cost function at each stage depends only on the current state-action pair, permits a drastic simplification --- at each time $t$, we can limit our search to {\em memoryless} channels $W_t(\d z_t|x_t)$ without impacting either the expected cost in \eqref{eq:baby_RI_AC} or the information constraint in \eqref{eq:baby_RI_info}:

\begin{theorem}[Memoryless observation channels suffice]\label{thm:structural} For any controller specification $\bd{\Phi}$ and any channel specification $\bd{W}$, there exists another channel specification $\bd{W}'$ consisting of stochastic kernels $W_t(\d z_t|x_t)$, $t=1,2,\ldots$, such that
	\begin{align*}
		\E\left[ \sum^T_{t=1} c(X'_t,U'_t)\right] = \E\left[\sum^T_{t=1} c(X_t,U_t)\right] \quad \text{and} \quad
		I(X'_t; Z'_t) = I(X_t; Z_t),\, t = 1,2,\ldots,T
	\end{align*}
	where $\{(X_t,U_t,Z_t)\}$ is the original process with $(\mu,Q,\bd{W},\bd{\Phi})$, while $\{X'_t,U'_t,Z'_t)\}$ is the one with $(\mu,Q,\bd{W}',\bd{\Phi})$.
\end{theorem}

\begin{proof} To prove the theorem, we follow the approach used by Wistenhausen in \cite{Witsenhausen1979realtime}. We start with the following simple observation that can be regarded as an instance of the Shannon--Mori--Zwanzig Markov model \cite{matmey08}:
	
\begin{lemma}[Principle of Irrelevant Information] Let $\Xi,\Theta,\Psi,\Upsilon$ be four random variables defined on a common probability space, such that $\Upsilon$ is conditionally independent of $(\Theta,\Xi)$ given $\Psi$. Then there exist four random variables $\Xi',\Theta',\Psi',\Upsilon'$ defined on the same spaces as the original tuple, such that 
$\Xi'\to\Theta'\to\Psi'\to\Upsilon'$ is a \textit{Markov chain}, and moreover the bivariate marginals agree:
\begin{align*}
\law(\Xi,\Theta) =\law(\Xi',\Theta'),\,\, \law(\Theta,\Psi) = \law(\Theta',\Psi'),\,\,
\law(\Psi,\Upsilon) =\law(\Psi',\Upsilon').
\end{align*}
\end{lemma}
\vspace{-10pt}
\begin{proof}If we denote by $M(\d \upsilon|\psi)$ the conditional distribution of $\Upsilon$ given $\Psi$ and by $\Lambda(\d \psi|\theta,\xi)$ be the conditional distribution of $\Psi$ given $(\theta,\xi)$, then we can disintegrate the joint distribution of $\Theta,\Xi,\Psi,\Upsilon$ as
	\begin{align*}
		P(\d\theta,\d\xi,\d\psi,\d\upsilon) = P(\d\theta)P(\d\xi|\theta)\Lambda(\d\psi|\theta,\xi)M(\d\upsilon|\psi).
	\end{align*}
	If we define $\Lambda'(\d\psi | \theta)$ by $\Lambda'(\cdot|\theta) = \int \Lambda(\cdot|\theta,\xi) P(\d \xi |\theta)$, and let the tuple $(\Theta',\Xi',\Psi',\Upsilon')$ have the joint distribution
	\begin{align*}
		P'(\d\theta,\d\xi, \d\psi, \d\upsilon) = P(\d\theta) P(\d\xi | \theta) \Lambda'( \d\psi | \theta) M (\d\upsilon | \psi),
	\end{align*}
	then it is easy to see that it has all of the desired properties.
\end{proof}

\noindent 
Using this principle, we can prove the following two lemmas:

\sloppypar\begin{lemma}[Two-Stage Lemma]Suppose $T=2$. Then the kernel $W_2(\d z_2|x^2,z_1,u_1)$ can be replaced by another kernel $W'_2(\d z_2|x_2)$, such that the resulting variables $(X'_t,Z'_t,U'_t)$, $t=1,2$, satisfy
	\begin{align*}
		\E[c(X'_1,U'_1) + c(X'_2,U'_2)] = \E[c(X_1,U_1) + c(X_2,U_2)]
	\end{align*}
	and $I(X_t'; Z'_t) = I(X_t; Z_t)$, $t = 1,2$.
\end{lemma}
\begin{proof}Note that $Z_1$ only depends on $X_1$, and that only the second-stage expected cost is affected by the choice of $W_2$. We can therefore apply the Principle of Irrelevant Information to $\Theta = X_2$, $\Xi = (X_1,Z_1,U_1)$, $\Psi = Z_2$ and $\Upsilon = U_2$. Because both the expected cost $\E[c(X_t,U_t)]$ and the mutual information $I(X_t; Z_t)$ depend only on the corresponding bivariate marginals, the lemma is proved. \end{proof}

\begin{lemma}[Three-Stage Lemma] Suppose $T=2$, and $Z_3$ is conditionally independent of $(X_i,Z_i,U_i)$, $i=1,2$, given $X_3$. Then the kernel $W_2(\d z_2|x^2,z_1,u_1)$ can be replaced by another kernel $W'_2(\d z_2|x_2)$, such that the resulting variables $(X'_i,Z'_i,U'_i)$, $i=1,2,3$, satisfy
	\begin{align*}
		\E\left[\sum^3_{t=1} c(X'_t,U'_t)\right] = \E\left[\sum^3_{t=1}c(X_t,U_t)\right]
	\end{align*}
	and
	$I(X'_t; Z'_t) = I(X_t; Z_t)$ for $t=1,2,3$.
\end{lemma}
\begin{proof}Again, $Z_1$ only depends on $X_1$, and only the second- and the third-stage expected costs are affected by the choice of $W_2$. By the law of iterated expectation, 
	\begin{align*}
	\E[ c(X_3,U_3)] &= \E[ \E[ c(X_3,U_3) | X_2,U_2]] = \E[h(X_2,U_2)],
	\end{align*}
	where the functional form of $h(X_2,U_2) \deq  \E[c(X_3,U_3) | X_2,U_2]$ is independent of the choice of $W_2$, since for any fixed realizations $X_2 = x_2$ and $U_2 = u_2$ we have
	\begin{align*}
	& h(x_2,u_2) = \int c(x_3,u_3) P(\d x_3, \d u_3 | x_2,u_2) \\
	&= \int c(x_3,u_3) Q(\d x_3 | x_2,u_2) W_3(\d z_3 | x_3) \Phi_3(\d u_3 | \d z_3),
	\end{align*}
by hypothesis.	Therefore, applying the Principle of Irrelevant Information to $\Theta = X_2$, $\Xi = (X_1,Z_1,U_1)$, $\Psi = Z_2$, and $\Upsilon = U_2$, 
	\begin{align*}
		\E[c(X'_2,U'_2) + c(X'_3,U'_3)] &= \E[c(X'_2,U'_2) + h(X'_2,U'_2)] \\
		&= \E[c(X_2,U_2) + h(X_2,U_2)] \\
		&= \E[c(X_2,U_2) + c(X_3,U_3)],
	\end{align*}
	where the variables $(X'_t,Z'_t,U'_t)$ are obtained from the original ones by replacing $W_2(\d z_2|x^2,z_1,u_1)$ by $W'_2(\d z_2|x_2)$. \end{proof}

\noindent Armed with these two lemmas, we can now prove the theorem by backward induction and grouping of variables. Fix any $T$. By the Two-Stage-Lemma, we may assume that $W_T$ is memoryless, i.e., $Z_T$ is conditionally independent of $X^{T-1},Z^{T-1},U^{T-1}$ given $X_T$. Now we apply the Three-Stage Lemma to \begin{align}
 &\Big|\underbrace{X^{T-3},Z^{T-3},U^{T-3},X_{T-2}}_{\text{Stage 1} \atop \text{state}},\underbrace{Z_{T-2}}_{\text{Stage 1} \atop \text{observation}},\underbrace{U_{T-2}}_{\text{Stage 1}\atop \text{control}}\Big| \nonumber \\
& \qquad\qquad \Big| \underbrace{X_{T-1}}_{\text{Stage 2} \atop \text{state}},\underbrace{Z_{T-1}}_{\text{Stage 2} \atop \text{observation}},\underbrace{U_{T-1}}_{\text{Stage 2}\atop \text{control}} \Big| \underbrace{X_T}_{\text{Stage 3} \atop \text{state}},\underbrace{Z_T}_{\text{Stage 3} \atop \text{observation}},\underbrace{U_T}_{\text{Stage 3}\atop \text{control}}\Big| \label{eq:induction}
\end{align}
to replace $W_{T-1}(\d z_{T-1}|x^{T-1},z^{T-2},u^{T-2})$ with $W'_{T-1}(\d z_{T-1}|x_{T-1})$ without affecting the expected cost or the mutual information between the state and the observation at time $T-1$. We proceed inductively by merging the second and the third stages in \eqref{eq:induction}, splitting the first stage in \eqref{eq:induction} into two, and then applying the Three-Stage Lemma to replace the original observation kernel $W_{T-2}$ with a memoryless one.
\end{proof}

	\section{The Gaussian distortion-rate function}
\label{ssec:gaussian}

Given a Borel probability measure $\mu$ on the real line, we denote by $D_\mu(R)$ its distortion-rate function w.r.t.\ the squared-error distortion $d(x,x') = (x-x')^2$:
\begin{align}\label{eq:quadratic_DRF}
	D_\mu(R) \deq \inf_{\small K \in \cM(\R | \R):\atop I(\mu,K) \le R} \int_{\R \times \R}(x-x')^2 \mu(\d x)K(\d x'|x)
\end{align}
Let $\mu = N(0,\sigma^2)$. Then we have the following \cite{Ratedistortion}: the DRF is equal to $D_\mu(R) = \sigma^2 e^{-2R}$; the optimal kernel $K^*$ that achieves the infimum in \eqref{eq:quadratic_DRF} has the form
	\begin{align}
		 K^*(\d x'|x) = \gamma\left(x'; (1-e^{-2R})x, (1-e^{-2R})e^{-2R}\sigma^2\right) \d x'.\label{eq:Gaussian_DRF_kernel}
	\end{align}r
Moreover, it achieves the information constraint with equality, $I(\mu,K^*) = R$, and can be realized as a stochastic linear system
	\begin{align}
		X' &= (1-e^{-2R})X + e^{-R}\sqrt{1-e^{-2R}}V,\label{eq:Gaussian_DRF_FC}
	\end{align}
	where $V \sim N(0,\sigma^2)$ is independent of $X$.

\end{appendix}

\section*{Acknowledgments} Several discussions with T.~Ba\c{s}ar, V.S.~Borkar, T.~Linder, S.K.~Mitter, S.~Tatikonda, and S.~Y\"uksel are gratefully acknowledged. The authors would also like to thank two anonymous referees for their incisive and constructive comments on the original version of the manuscript.

\end{document}